\documentclass{article}

\usepackage{arxiv}

\usepackage[utf8]{inputenc}
\usepackage[T1]{fontenc}
\usepackage[english]{babel} 

\usepackage{amsmath,amssymb,amsthm,mathtools}
\usepackage{mathrsfs}
\usepackage{bm}
\usepackage{dsfont}
\usepackage{bbm}

\usepackage{graphicx}
\usepackage{epsfig}

\usepackage{xcolor}        
\usepackage{hyperref}
\usepackage{cleveref}      
\usepackage{url}
\usepackage{booktabs}

\usepackage{microtype}
\usepackage{enumitem}      
\usepackage{comment}
\usepackage{csquotes}      

\usepackage{fouriernc}     

\hypersetup{
	colorlinks=true,
	linkcolor=blue,
	citecolor=blue,
	urlcolor=blue,
}

\newtheorem{theorem}{Theorem}[section]
\newtheorem{lemma}[theorem]{Lemma}
\newtheorem{definition}[theorem]{Definition}
\theoremstyle{remark} \newtheorem{remark}{Remark}

\title{Stochastic Optimal Control with Control-Dependent Diffusion and State Constraints: A Degenerate Elliptic Approach}

  \author{
	Anderson O. Calixto\thanks{The author would like to thank FAPERJ grant E-26/202.535/2022} \\
	Departamento de Estatística \\
	Universidade Federal Fluminense \\
	\texttt{acalixto@id.uff.br}
	\And
	Bernardo Freitas Paulo da Costa \\
	School of Applied Mathematics  \\
	Getulio Vargas Foundation\\
	\texttt{bernardo.paulo@fgv.br} \\
	\And
	Glauco Valle \thanks{The author would like to thank FAPERJ grant E-26/202.636/2019 and CNPq grants 307938/2022-0 and 403423/20236}
	\\
	Instituto de Matemática \\
	Universidade Federal do Rio de Janeiro\\
	\texttt{glauco.valle@im.ufrj.br} \\
}

\begin{document}
\maketitle
		\begin{abstract}
           We study a stochastic optimal control problem with the state constrained to a smooth, compact domain. The control influences both the drift and a possibly degenerate, control-dependent dispersion matrix, leading to a fully nonlinear, degenerate elliptic Hamilton--Jacobi--Bellman (HJB) equation with a nontrivial Neumann boundary condition. Although these features have been studied separately, this work provides the first unified treatment combining them all. We establish that the optimal value function associated with the control problem is the unique viscosity solution of the HJB equation with a nontrivial Neumann boundary condition, and we present an illustrative example demonstrating the applicability of the framework.
		\end{abstract}
	\vspace{1em}
	\noindent\textbf{Keywords:} Stochastic Optimal Control, Hamilton–Jacobi–Bellman equation, Degenerate Elliptic PDEs, Viscosity Solutions, State Constraints, Neumann Boundary Condition.
\section{Introduction}

Stochastic Optimal Control (SOC) arises in various applications, including mathematical finance~\cite{Pham, Carmona}, engineering~\cite{PengLi}, and economics~\cite{Morimoto}. Traditional approaches typically assume uniformly non-degenerate diffusions and unconstrained domains, which simplify the analysis of the associated Hamilton--Jacobi--Bellman (HJB) equations. However, real-world problems often involve explicit state constraints and control-influenced noise structures that may be degenerate, thereby complicating both the theoretical and numerical analysis.

When state constraints are present, it becomes necessary to impose suitable boundary conditions on the HJB equation. A foundational contribution by Lions and Soner~\cite{LionsSoner} introduced Neumann-type boundary conditions within the viscosity solution framework, enabling the treatment of constrained problems in compact domains. The viscosity solution theory developed by Crandall, Ishii, and Lions~\cite{Crandall} provides a robust and flexible framework for analyzing fully nonlinear PDEs, including degenerate and control-dependent cases.

Degenerate diffusions --- where the diffusion matrix is only positive semidefinite or may vanish in certain directions --- arise naturally in applications such as portfolio optimization~\cite{Pham}. While key works by Fleming and Soner~\cite{Soner}, Nisio~\cite{Nisio}, and Krylov~\cite{Krylov} analyze degenerate controlled diffusions, they generally do not address state constraints in an explicit or systematic manner.

In the context of ergodic control, the comprehensive treatment by Arapostathis, Borkar, and Ghosh~\cite{Ghosh} considers diffusion processes under a long-term average cost criterion, including controlled reflected diffusions on compact domains, where the state process is confined via boundary reflection. In this setting, the control affects only the drift, and the diffusion matrix is uniformly positive definite, ensuring the ellipticity of the HJB equation. This structure allows for the existence of classical solutions under homogeneous Neumann boundary conditions---i.e., vanishing directional derivatives at the boundary. While restrictive, this framework provides a rigorous foundation for incorporating state constraints in ergodic control problems.

In a complementary direction, Kushner and Dupuis~\cite{Kushner} outline a viscosity solution approach to stochastic control problems with state constraints, focusing on models in which the control acts solely on the drift and homogeneous Neumann boundary conditions are imposed. Although their treatment is not exhaustive, it provides important insights into extending viscosity methods to constrained settings.

Despite these significant advances, a unified framework that simultaneously accounts for state constraints, degenerate and control-dependent diffusions, and nontrivial boundary conditions is still lacking. This paper aims to bridge this gap by formulating a general stochastic control problem on a compact domain, in which the control influences both the drift and the (possibly degenerate and control-dependent) diffusion. We establish existence and uniqueness results for viscosity solutions of the corresponding HJB equation under nontrivial boundary conditions.

Furthermore, this paper lays the foundation for the framework developed by the authors in~\cite{CalixtoCostaValle2025}, which enables the analysis of singular perturbation problems in multiscale stochastic optimal control, where the slow variable is confined to a convex, compact domain with a smooth boundary.

This paper is organized as follows. In Section~\S\ref{sec:stochastic_optimal_control_restrictions_state_variables}, we formulate a SOC problem for a system with a positive semidefinite, control-dependent diffusion matrix, subject to state constraints. The constraints are enforced via a penalty function that penalizes deviations from a prescribed compact domain. In Subsection~\S\ref{subsec:continuity_optimal_value_function}, we establish the continuity of the value function, a key ingredient for applying viscosity solution techniques. Subsection~\S\ref{subsec:dynamic_programming_principle} is devoted to proving the Dynamic Programming Principle in the context of strong solutions to the controlled reflected stochastic system. In Subsection~\S\ref{subsec:viscosity_solution_HJB_equation_von_Neumann_boundary_condition}, we prove the main result: the value function is a viscosity solution of a degenerate elliptic HJB equation with nontrivial Neumann boundary conditions. 

Finally, in Section~\S\ref{sec:examples}, we present two illustrative examples. The first, discussed in Subsection~\S\ref{subsec:control_problem_semilinear_hjb_equation}, considers a control problem with state constraints and a positive semidefinite diffusion matrix, allowing for a detailed analysis of the associated optimal control. The second example,  discussed in Subsection~\S\ref{subsec:control_problem_fully_nonlinear_HJB_equation}, introduces control directly into the diffusion matrix, resulting in a problem that captures all the essential features addressed by the proposed framework and, to the best of our knowledge, falls outside the scope of existing approaches in the literature.

\section{Stochastic Optimal Control with State Constraints}
\label{sec:stochastic_optimal_control_restrictions_state_variables}

In this section, we analyze a SOC problem in which a constraint is imposed on the state variable. We associate this constraint with a cost function that quantifies the effort --- or cost --- required to keep the state within the desired bounds. At the same time, we allow the dispersion matrix of the underlying Stochastic Differential Equation (SDE), whose states we aim to control, to depend on the control variable.

We consider a stochastic basis \(\big(\Omega, \mathcal{O}, (\mathcal{F}_t)_{t \geq 0}, \mathbb{P}\big)\) satisfying the usual conditions; that is, the filtration \((\mathcal{F}_t)_{t \geq 0}\) is right-continuous and contains all \(\mathbb{P}\)-null sets. Furthermore, \((W(t))_{t \geq 0}\) denotes a standard \(d_W\)-dimensional Brownian motion defined on this basis. In the context of strong solutions to SDEs, the filtration \((\mathcal{F}_t)_{t \geq 0}\) is taken to be the one generated by the Brownian motion, augmented with the \(\mathbb{P}\)-null sets.

We also fix a compact set \(\mathbb{X} \subset \mathbb{R}^d\), which plays a central role in the formulation of the state constraint. Specifically, we assume the existence of a function \(\phi \in \text{C}^2_{\text{b}}\left(\mathbb{R}^d\right)\) such that:
\begin{itemize}
	\item \(\mathbb{X} = \left\{x \in \mathbb{R}^d : \phi\big(x\big) \leqslant 0\right\}\);
	\item \(\text{Int}\left(\mathbb{X}\right) = \left\{x \in \mathbb{R}^d : \phi\big(x\big) < 0\right\}\);
	\item \(\partial \mathbb{X} = \left\{x \in \mathbb{R}^d : \phi\big(x\big) = 0\right\}\), with \(\|D_x\phi\big(x\big)\| = 1\) for all \(x \in \partial \mathbb{X}\).
\end{itemize}

The dynamical system we aim to control evolves within the constrained state space \(\mathbb{X}\) and is governed by the solution to the following SDE with reflection:
 \begin{equation}
	\label{eq:sde_reflection_control}
    dX_x(t) = \mu_X\big(X_x(t), u(t)\big)\,dt + \sigma_X\big(X_x(t), u(t)\big)\,dW(t) - D_x\phi\big(X_x(t)\big)\,dl_x(t).
 \end{equation}
This equation captures the interplay between the stochastic dynamics of the system and the reflection mechanism that ensures the state remains within the set \(\mathbb{X}\). The system consists of the following components:
\begin{itemize}
	\item \(\mu_X\) and \(\sigma_X\) represent the drift and dispersion fields, respectively;
	\item The process \((X_x(t))_{t \geq 0}\) is continuous and adapted to \((\mathcal{F}_t)_{t \geq 0}\);
	\item The process \((l_x(t))_{t \geq 0}\) is continuous, non-decreasing, and adapted to \((\mathcal{F}_t)_{t \geq 0}\), with \(l_x(0) = 0\), and it satisfies the condition:
	\[
	l_x(t) = \int_0^t \mathds{1}_{\partial \mathbb{X}}\big(X_x(s)\big)\,dl_x(s) \quad \mathbb{P}\text{-a.s.}
	\]
	\item The process \((u(t))_{t \geq 0}\) is progressively measurable with respect to \((\mathcal{F}_t)_{t \geq 0}\), taking values in a convex and compact set \(\mathbb{U} \subset \mathbb{R}^m\). The set of such controls is denoted by \(\mathcal{U}\), and its elements are referred to as admissible controls.
\end{itemize}

\begin{remark}
	The intuitive interpretation of \(-D_x\phi\big(X_x(t)\big)\) is as follows: this gradient acts as a force field that reflects the process \((X_x(t))_{t \geq 0}\) back into the domain \(\mathbb{X}\). The reflection is instantaneously triggered by the process \((l_x(t))_{t \geq 0}\) whenever the state \(X_x(t)\) reaches the boundary \(\partial \mathbb{X}\).
\end{remark}

We now introduce the following hypothesis for the drift and dispersion fields:

\begin{itemize}
	\item[\(\textbf{H}_1\)] The drift \(\mu_{X}\) and the dispersion \(\sigma_{X}\) are continuous, bounded, and Lipschitz continuous with respect to the spatial variable, uniformly with respect to the control variable. Specifically, there exists a constant \(C > 0\) such that for every \(x, y \in \mathbb{R}^d\) and every \(u \in \mathbb{U}\), the following inequalities hold:
	\[
	\|\mu_{X}\big(x,u\big)-\mu_{X}\big(y,u\big)\|+\|\sigma_{X}\big(x,u\big)-\sigma_{X}\big(y,u\big)\| \leqslant C\|x-y\| \quad \text{and} \quad \|\mu_{X}\big(x,u\big)\| + \|\sigma_{X}\big(x,u\big)\| \leqslant C.
	\]
\end{itemize}

Under this hypothesis, it follows from \cite{Sznitman}, \cite{Pardoux}, and \cite{Pilipenko} that the SDE \eqref{eq:sde_reflection_control} admits a unique strong solution.

To complete the model formulation, we now define the cost functional. In particular, we distinguish between the operation cost \(L\) in the interior and the preventive cost \(h\) on the boundary.

\begin{itemize}
	\item[\(\textbf{H}_2\)] The function \(L: \mathbb{R}^d \times \mathbb{U} \to \mathbb{R}\) is continuous and bounded in \((x, u)\), and Lipschitz continuous in the spatial variable, uniformly in the control. That is, there exists a constant \(C > 0\) such that for every \(x, y \in \mathbb{R}^d\) and every \(u \in \mathbb{U}\),
	\[
	\big|L\big(x,u\big)-L\big(y,u\big)\big| \leqslant C\|x - y\| \quad \text{and} \quad \big|L\big(x, u\big)\big| \leqslant C.
	\]
	This function is referred to as the operation cost.
	
	\item[\(\textbf{H}_3\)] The function \(h: \partial \mathbb{X} \to \mathbb{R}\) is continuous. It is referred to as the preventive cost on the boundary.
\end{itemize}

Given \(\beta > 0\) and \(x \in \mathbb{X}\), we define the cost functional \(J^{\beta}_x: \mathcal{U} \to \mathbb{R}\) by
\begin{equation}
	\label{eq:operation_cost_functional}
	J^{\beta}_x\big(u\big) := \mathbb{E}\Bigg[\int_{0}^{+\infty} e^{-\beta s}L\big(X_x(s), u(s)\big)\,ds+\int_{0}^{+\infty} e^{-\beta s} h\big(X_x(s)\big)\,dl_x(s)\Bigg].
\end{equation}

With these definitions in place, we now present the first auxiliary result required for the analysis.

\begin{lemma}
	\label{lm:inequality_continuity_control_reflection}
	Given \(u \in \mathcal{U}\), assume that the pair of processes \(\big(X^{u}(t), l^{u}(t)\big)_{t \geq 0}\) is a strong solution of the SDE \eqref{eq:sde_reflection_control}. Then, for all \(p \geq 1\), \(T \geq 0\), and \(x, y \in \mathbb{X}\), the following inequalities hold:
	\begin{subequations}
		\begin{align}
			\label{eq:inequality_continuity_X_control_reflection}
			\mathbb{E}\Bigg[\sup_{r \in [0,T]}\|X^{u}_x(r)-X^{u}_y(r)\|^p\Bigg] &\leqslant C\exp\big\{CT\big\}\|x - y\|^p,\\
			\label{eq:inequality_upper_quota_X_control_reflection}
			\mathbb{E}\Bigg[\sup_{r \in [0,T]}\|X^{u}_x(r)\|^p\Bigg] &\leqslant C\left(1 + T^p + \|x\|^p\right),\\
			\label{eq:inequality_l_control_reflection}
			\mathbb{E}\Bigg[\sup_{r \in [0,T]}l^{u}_x(r)^p\Bigg] &\leqslant C\left(1 + T^p\right),
		\end{align}
	\end{subequations}
	where \(C > 0\) is a constant independent of \(u\), \(T\), \(x\), and \(y\).
\end{lemma}
\begin{proof}
	We apply \cite[Proposition 4.55]{Pardoux} to the pair \(\big(X^{u}(t), l^{u}(t)\big)_{t \geq 0}\), which is a strong solution of the SDE \eqref{eq:sde_reflection_control}. The independence of the constant \(C > 0\) from the control \(u \in \mathcal{U}\) follows from hypothesis \(\textbf{H}_1\), where the Lipschitz constant and the uniform bounds for the drift and diffusion coefficients do not depend on \(u\).
\end{proof}

\begin{lemma}
	\label{lm:functional_J_beta_well_defined}
	The functional \(J^{\beta}_{x}: \mathcal{U} \to \mathbb{R}\) is well defined and finite.
\end{lemma}
\begin{proof}
	We must check the integrability of the following terms:
	\[
	\mathbb{E}\Bigg[\int_{0}^{+\infty}e^{-\beta s}\big|L\big(X_x(s), u(s)\big)\big|\,ds\Bigg] < \infty \quad \text{and} \quad \mathbb{E}\Bigg[\int_{0}^{+\infty}e^{-\beta s}\big|h\big(X_x(s)\big)\big|\,dl_x(s)\Bigg] < \infty.
	\]
	By hypothesis \(\textbf{H}_2\), there exists a constant \(C_0 > 0\) such that \(\big|L\big(x,u\big)\big| \leqslant C_0\) for all \((x,u) \in \mathbb{R}^d \times \mathbb{U}\). Thus, by Fubini's theorem,
	\begin{equation}
		\label{eq:auxiliary_inequality_14}
		\mathbb{E}\Bigg[\int_{0}^{+\infty}e^{-\beta s}\big|L\big(X_x(s), u(s)\big)\big|\,ds\Bigg] = \int_{0}^{+\infty}e^{-\beta s}\mathbb{E}\left[\big|L\big(X_x(s), u(s)\big)\big|\right]\,ds \leqslant \frac{C_0}{\beta}.
	\end{equation}
	
	For the second term, hypothesis \(\textbf{H}_3\) implies that there exists \(C_0 > 0\) such that \(\big|h\big(x\big)\big| \leqslant C_0\) for all \(x \in \partial\mathbb{X}\). Using integration by parts and Fatou's lemma:
	\[\mathbb{E}\Bigg[\int_{0}^{+\infty}e^{-\beta s}\big|h\big(X_x(s)\big)\big|\,dl_x(s)\Bigg]
	\leqslant C_0\,\mathbb{E}\Bigg[\int_{0}^{+\infty}e^{-\beta s}\,dl_x(s)\Bigg]\leqslant C_0\,\liminf_{r \to +\infty} \mathbb{E}\Bigg[l_x(r)e^{-\beta r} + \beta \int_{0}^{r} l_x(s)e^{-\beta s}\,ds\Bigg].\]
	
	By inequality \eqref{eq:inequality_l_control_reflection} (with \(p=1\)):
	\begin{align*}
		\mathbb{E}\Bigg[\int_{0}^{+\infty}e^{-\beta s}\big|h\big(X_x(s)\big)\big|\,dl_x(s)\Bigg]
		&\leqslant C_0 C \lim_{r \to +\infty}\left(e^{-\beta r}\big(1 + r\big) + \beta \int_{0}^{r}e^{-\beta s}\big(1 + s\big)\,ds\right) \\
		&= C_0 C \beta \int_{0}^{+\infty}e^{-\beta s}\big(1 + s\big)\,ds < \infty.
	\end{align*}
	Therefore,
	\begin{equation}
		\label{eq:auxiliary_inequality_16}
		\mathbb{E}\Bigg[\int_{0}^{+\infty}e^{-\beta s}\big|h\big(X_x(s)\big)\big|\,dl_x(s)\Bigg] \leqslant C_0 C \beta \int_{0}^{+\infty}e^{-\beta s}\big(1 + s\big)\,ds.
	\end{equation}
	From inequalities \eqref{eq:auxiliary_inequality_14} and \eqref{eq:auxiliary_inequality_16}, we conclude that for all \(u \in \mathcal{U}\),
	\[
	\big|J^{\beta}_x\big(u\big)\big| \leqslant \frac{C_0}{\beta} + C_0 C \beta \int_{0}^{+\infty}e^{-\beta s}\big(1 + s\big)\,ds.
	\]
\end{proof}

\subsection{Continuity of the Optimal Value Function}
\label{subsec:continuity_optimal_value_function}

The control problem is formulated by minimizing the cost functional \eqref{eq:operation_cost_functional} over the set of admissible controls \(\mathcal{U}\). To begin its solution, we define the optimal value function as follows:
\begin{equation}
	\label{eq:optimal_value_function}
	v^{\beta}\big(x\big) := \inf_{u \in \mathcal{U}} J^{\beta}_{x}\big(u\big).
\end{equation}

Our goal is to prove that this function is at least continuous in the domain \(\mathbb{X}\). Such continuity is required to ensure the validity of the dynamic programming principle, which will be established in Section~\S\ref{subsec:dynamic_programming_principle}, and to enable the application of the viscosity solution concept in Section~\S\ref{subsec:viscosity_solution_HJB_equation_von_Neumann_boundary_condition}. To this end, we introduce some notations and definitions.

We say that two processes, \(\big(X(t)\big)_{t \geq 0}\) and \(\big(\widetilde{X}(t)\big)_{t \geq 0}\), are equivalent if \(X\) is a modification of \(\widetilde{X}\), that is, for all \(t \in \mathbb{R}_{+}\), we have \(X(t) = \widetilde{X}(t)\) \(\mathbb{P}\)-almost surely (i.e., \(\mathbb{P}\)-a.s.).

\begin{definition}
	\label{df:space_S_p_d_T}
	Given \(T > 0\), \(d \in \mathbb{N}\), and \(p \in [1,+\infty)\), we denote by \(\mathcal{S}^{\emph{p}}_{\emph{d}}\big([0,T]\big)\) the space of equivalence classes of processes \(X: [0,T] \times \Omega \to \mathbb{R}^d\) that are progressively measurable and satisfy
	\[
	\mathbb{E}\Bigg[\sup_{t \in [0,T]}\|X(t)\|^p\Bigg] < \infty.
	\]
\end{definition}

The space \(\mathcal{S}^{\text{p}}_{\text{d}}\big([0,T]\big)\) can naturally be equipped with the following norm:
\begin{equation}
	\label{eq:norm_S_p_d}
	\|X\|_{\mathcal{S}^{\text{p}}_{\text{d}}\big([0,T]\big)} := \mathbb{E}\Bigg[\sup_{t \in [0,T]}\|X(t)\|^p\Bigg]^{\frac{1}{p}}.
\end{equation}

\begin{definition}
	\label{df:space_S_p_d_inf}
	We denote by \(\mathcal{S}^{\emph{p}}_{\emph{d}}\big(\mathbb{R}_{+}\big)\) the space of equivalence classes of processes \(X: \mathbb{R}_{+} \times \Omega \to \mathbb{R}^d\) that are progressively measurable and such that for all \(T > 0\), the restriction \(X|_{[0,T]} \in \mathcal{S}^{\emph{p}}_{\emph{d}}\big([0,T]\big)\).
\end{definition}

For the space \(\mathcal{S}^{\text{p}}_{\text{d}}\big(\mathbb{R}_{+}\big)\), we consider the following norm:
\begin{equation}
	\label{eq:norm_space_S_p_d_inf}
	\|X\|_{\mathcal{S}^{\text{p}}_{\text{d}}\big(\mathbb{R}_{+}\big)} := \int_{0}^{+\infty} e^{-t} \min\Bigg\{1, \|X\|_{\mathcal{S}^{\text{p}}_{\text{d}}\big([0,t]\big)}\Bigg\}\, dt.
\end{equation}
\begin{lemma}[Equicontinuity of the family \(\big(x\mapsto\big(X^{u}_x(\cdot),l^{u}_x(\cdot)\big)\big)_{u\in\mathcal{U}}\)]
	\label{lm:continuity_processes_X_K_VarK}
	Consider that the pair of processes \(\big(X^{u}_x(t),l^{u}_x(t)\big)_{t\geq0}\) is the unique strong solution of the SDE \eqref{eq:sde_reflection_control}. Then, for all \(p\geq1\), the application
	\begin{equation}
		\label{eq:auxiliary_application_X_l}
		\mathbb{X}\ni x\mapsto \big(X^{u}_x(\cdot),l^{u}_x(\cdot)\big)\in \mathcal{S}^{\emph{p}}_{\emph{d}}\big(\mathbb{R}_{+}\big) \times \mathcal{S}^{\emph{p}}_1\big(\mathbb{R}_{+}\big)
	\end{equation}
	is equicontinuous with respect to $u\in\mathcal{U}$. 
\end{lemma}
\begin{proof}
	From the inequalities \eqref{eq:inequality_upper_quota_X_control_reflection} and \eqref{eq:inequality_l_control_reflection} we get that the application \eqref{eq:auxiliary_application_X_l} is well defined. Given \(x,y\in\mathbb{X}\) and \(X^{u}_x(\cdot),X^{u}_y(\cdot)\in \mathcal{S}^{\text{p}}_{\text{d}}\big(\mathbb{R}_{+}\big)\) it follows from definition \ref{df:space_S_p_d_inf} that, for all \(t\geq0\) the restrictions of these processes to the interval \([0,t]\) belong to the space \(\mathcal{S}^{\text{p}}_{\text{d}}\big([0,t]\big)\). On the other hand, we have from the inequality \eqref{eq:inequality_continuity_X_control_reflection} in lemma \ref{lm:inequality_continuity_control_reflection} that
	\begin{equation*}
		\|X^{u}_x-X^{u}_y\|_{\mathcal{S}^{\emph{p}}_{\emph{d}}\big([0,t]\big)}\leqslant \sqrt[p]{Ce^{Ct}}\|x-y\|
	\end{equation*}
	where \(C>0\) is a constant independent of \(u\in\mathcal{U}\). By definition \ref{eq:norm_space_S_p_d_inf} 
	\begin{equation}
		\label{eq:auxiliary_inequality_01}
		\sup_{u\in \mathcal{U}}\|X^{u}_x-X^{u}_y\|_{\mathcal{S}^{\emph{p}}_{\emph{d}}\big(\mathbb{R}_{+}\big)}\leqslant\int_{0}^{+\infty}e^{-t}\min\Bigg\{1,\sqrt[p]{Ce^{Ct}}\|x-y\|\Bigg\}\,dt.
	\end{equation}
	Therefore, by the dominated convergence theorem,  the family \(\big(x\mapsto X^{u}_x(\cdot)\big)_{u\in\mathcal{U}}\) is equicontinuous. 
	
	For \(x\mapsto l^{u}_x(\cdot)\), we apply Itô's formula to the function \(\phi\in \text{C}^2_{\text{b}}\big(\mathbb{R}^d\big)\) which comes from the representation of the domain \(\mathbb{X}\) and get 
	\begin{align}
		\nonumber
		\phi\big(X^{u}_x(\cdot)\big)=\phi\big(x\big)+&\int_{0}^{\cdot}\sum_{i=1}^{d}\partial_{x_i}\phi\big(X^{u}_x(s)\big)\mu_{X,i}\big(X^{u}_x(s),u(s)\big)\,ds-\int_{0}^{\cdot}\|D_x\phi\big(X^{u}_x(s)\big)\|^2\,dl^{u}_x(s)\\
		\label{eq:auxiliary_equation_07}
		+&\int_{0}^{\cdot}\frac{1}{2}\sum_{i=1}^{d}\sum_{j=1}^{d}\partial^2_{x_i,x_j}\phi\big(X^{u}_x(s)\big)\big[\sigma_{X}\sigma^{\top}_{X}\big]_{ij}\big(X^{u}_x(s),u(s)\big)\,ds\\
		\nonumber
		+&\int_{0}^{\cdot}\sum_{i=1}^{d}\sum_{l=1}^{d_W}\partial_{x_i}\phi\big(X^{u}_x(s)\big)\sigma_{X,il}\big(X^{u}_x(s),u(s)\big)\,dW^{(l)}(s)\quad \mathbb{P}\text{-a.s}.
	\end{align}	
	Since the process \(\big(l^{u}_x(t)\big)_{t\geq0}\) satisfies the equation
	\begin{equation*}
		l^{u}_x(t)=\int_{0}^{t}\mathds{1}_{\partial \mathbb{X}}\big(X^{u}_x(s)\big)\,dl^{u}_x(s)\quad \mathbb{P}\text{-a.s},
	\end{equation*}
	the third term on the right-hand side in the equality \eqref{eq:auxiliary_equation_07} can be written as 
	\begin{equation*}
		l^{u}_x(t)=\int_{0}^{t}\|D_x\phi\big(X^{u}_x(s)\big)\|^2 \mathds{1}_{\partial \mathbb{X}}\big(X^{u}_x(s)\big)\,dl^{u}_x(s)\quad \mathbb{P}\text{-a.s}.
	\end{equation*}
	Where equality follows by the representation of the domain \(\mathbb{X}\) and \(\|D_x\phi\big(x\big)\|=1\) for all \(x\in \partial\mathbb{X}\). Denoting by \(\mathcal{L}^{u}_X:\text{C}_{\text{b}}^2\big(\mathbb{R}^d\big)\to \text{C}\big(\mathbb{R}^d\big)\) the differential operator defined by
	\begin{equation}
		\label{eq:auxiliary_equation_08}
		\mathcal{L}^{u}_X\varphi\big(x\big):=\bigg\langle \mu_X\big(x,u\big),D_x\varphi\big(x\big)\bigg\rangle+\frac{1}{2}\text{Tr}\Bigg(\big[\sigma_{X}\sigma^{\top}_{X}\big]\big(x,u\big)D^2_x\varphi\big(x\big)\Bigg)
	\end{equation}
	it follows from \eqref{eq:auxiliary_equation_07} that
	\begin{equation*}
		l^{u}_x\big(\cdot\big)=\int_{0}^{\cdot}\mathcal{L}^{u}_X\phi\big(X^{u}_x(s)\big)\,ds+\sum_{i=1}^d\sum_{l=1}^{d_W}\int_{0}^{\cdot}\partial_{x_i}\phi\big(X^{u}_x(s)\big)\sigma_{X,il}\big(X^{u}_x(s),u(s)\big)\,dW^{(l)}(s)-\bigg(\phi\big(X^{u}_x\big(\cdot\big)\big)-\phi\big(x\big)\bigg)\quad \mathbb{P}\text{-a.s}.
	\end{equation*}
	On the other hand, note that it is possible to develop the term
	\begin{equation}
		\label{eq:auxiliary_equality_02}
		x\mapsto\int_{0}^{\cdot}\mathcal{L}^{u}_X\phi\big(X^{u}_x(s)\big)\,ds+\sum_{i=1}^d\sum_{l=1}^{d_W}\int_{0}^{\cdot}\partial_{x_i}\phi\big(X^{u}_x(s)\big)\sigma_{X,il}\big(X^{u}_x(s),u(s)\big)\,dW^{(l)}(s),
	\end{equation}
	using classical arguments from stochastic calculus (such as Itô isometry), combined with hypothesis \(\textbf{H}_1\) and inequality \eqref{eq:auxiliary_inequality_01} to obtain the  equicontinuity of \eqref{eq:auxiliary_equality_02} in \(\mathcal{S}^{\text{p}}_1\big(\mathbb{R}_{+}\big)\) with respect to \(u\in\mathcal{U}\). Thus, we obtain that the family \(\big(x\mapsto l^{u}_x\big(\cdot\big)\big)_{u\in\mathcal{U}}\) is equicontinuous.
\end{proof}
We now show that the family \(\big(x\mapsto J^{\beta}_x\big(u\big)\big)_{u\in\mathcal{U}}\) is equicontinuous.
\begin{lemma}[Equicontinuity of the family \(\big(x\mapsto J^{\beta}_x\big(u\big)\big)_{u\in\mathcal{U}}\)]
	\label{lm:continuity_functional_cost_initial_condition}
	Given \(\beta>0\) the family of applications \(\big(x\mapsto J^{\beta}_x\big(u\big)\big)_{u\in\mathcal{U}}\) is equicontinuous.
\end{lemma}
\begin{proof}
	Given $x_0\in \mathbb{X}$, consider $\big(x_n\big)_{n\geqslant0}$ a sequence in $\mathbb{X}$ such that $\lim_{n\to \infty}x_n=x_0$. Let's prove that 
	\begin{equation}
		\label{eq:limit_continuity_J}
		\lim_{n\to \infty}\sup_{u\in \mathcal{U}}\big|J^{\beta}_{x_n}\big(u\big)-J^{\beta}_{x_0}\big(u\big)\big|=0.
	\end{equation}
	From the definition of $J^{\beta}_x$ we can write 
	\begin{equation*}
		J^{\beta}_{x_n}\big(u\big)=\mathbb{E}\big[L^{u}_n\big]
		+\mathbb{E}\big[H^{u}_n\big],
	\end{equation*}
	where:
	\begin{equation*}
		L^{u}_n:=\int_{0}^{+\infty}e^{-\beta s}L\big(X^{u}_{x_n}(s),u(s)\big)\,ds\quad\text{and}\quad H^{u}_n:=\int_{0}^{+\infty}e^{-\beta s}h\big(X^{u}_{x_n}(s)\big)\,dl^{u}_{x_n}(s).
	\end{equation*}
	We will divide the proof of the limit \eqref{eq:limit_continuity_J} into several steps.
	\begin{itemize}
		\item[Step 1](The family $\big\{H^{u}_n\big\}_{n\geqslant0}$ is uniformly integrable.) 
		By an analogous argument to the one used in the proof of Lemma \ref{lm:functional_J_beta_well_defined}, we have that
		\begin{equation}
			\label{eq:auxiliary_inequality_05}
			\mathbb{E}\big[\big|H^{u}_n\big|^2\big]\leqslant 2\beta^2C_0 C^2\int_{0}^{+\infty}e^{-2\beta s}\big(1+s^2\big)ds=:\gamma<\infty. 	
		\end{equation}
		Where the bound we obtain is independent of $u\in \mathcal{U}$. The above inequality implies that
		\begin{equation*}
			\sup_{n\geqslant0}\mathbb{E}\big[\big|H^{u}_n\big|^2\big]<\infty.
		\end{equation*} 
		Thus, from \cite[Item 2 of Lemma 1.12]{Pardoux}  (with $p=2$), the family $\big\{H^{u}_n\big\}_{n\geqslant0}$ is uniformly integrable. 
		
		To finish the proof, we need to show that 
		\begin{equation*}
			\lim_{n\to+\infty}\sup_{u\in \mathcal{U}}\big|\mathbb{E}\big[L^{u}_n\big]- \mathbb{E}\big[L^{u}_{\infty}\big]\big|=0\quad\text{and}\quad\lim_{n\to+\infty}\sup_{u\in \mathcal{U}}\big|\mathbb{E}\big[H^{u}_n\big]-\mathbb{E}\big[H^{u}_{\infty}\big]\big|=0
		\end{equation*} 
		where 
		\begin{equation*}
			L^{u}_{\infty}:=\int_{0}^{+\infty}e^{-\beta s}L\big(X^{u}_{x_0}(s),u(s)\big)ds\quad\text{and}\quad
			H^{u}_{\infty}:=\int_{0}^{+\infty}e^{-\beta s}h\big(X^{u}_{x_0}(s)\big)dl^{u}_{x_0}(s).
		\end{equation*}
		\item[Step 2] In the first case, ($\big(L^{u}_n\big)_{n\geqslant0}$ converging to $L^{u}_{\infty}$), consider the auxiliary sequences $\big(L^{u}_{n,T}\big)_{(n,T)\in \mathbb{N}\times\mathbb{R}_{+}}$ and $\big(L^{u}_{\infty,T}\big)_{T\in \mathbb{R}_{+}}$ defined respectively by:
		\begin{equation*}
			L^{u}_{n,T}:=\int_{0}^{T}e^{-\beta s}L\big(X^{u}_{x_n}(s),u(s)\big)ds\quad\text{and}\quad
			L^{u}_{\infty,T}:=\int_{0}^{T}e^{-\beta s}L\big(X^{u}_{x_0}(s),u(s)\big)ds
		\end{equation*}
		Based on these sequences, we can write
		\begin{equation}
			\label{eq:auxiliary_inequality_06}
			\big|L^{u}_n-L^{u}_{\infty}\big|\leqslant \big|L^{u}_n-L^{u}_{n,T}\big|+\big|L^{u}_{n,T}-L^{u}_{\infty,T}\big|+\big|L^{u}_{\infty,T}-L^{u}_{\infty}\big|\quad \mathbb{P}\text{-a.s}.
		\end{equation}
		Developing the first part on the right-hand side and recalling that according to hypothesis $\textbf{H}_2$ the operation cost is bounded, it follows that
		\begin{equation*}
			\big|L^{u}_n-L^{u}_{n,T}\big|=\Bigg|\int_{T}^{+\infty}e^{-\beta s}L\big(X^{u}_{x_n}(s),u(s)\big)ds\Bigg|\leqslant\int_{T}^{+\infty}Ce^{-\beta s}ds=\frac{Ce^{-T\beta}}{\beta}\quad \mathbb{P}\text{-a.s}.
		\end{equation*}
		Applying expectation to both sides of the above inequality, we get
		\begin{equation*}
			\mathbb{E}\big[\big|L^{u}_n-L^{u}_{n,T}\big|\big]\leqslant \frac{Ce^{-T\beta}}{\beta}.
		\end{equation*}
		Then, given $\epsilon>0$ there exists $T_1>0$ such that for every $T\geqslant T_1$ 
		\begin{equation}
			\label{eq:auxiliary_inequality_07}
			\sup_{u\in \mathcal{U}}\mathbb{E}\big[\big|L^{u}_n-L^{u}_{n,T}\big|\big]\leqslant \frac{Ce^{-T\beta}}{\beta}<\frac{\epsilon}{3}.
		\end{equation}
		By an analogous argument, there exists $T_2>0$ 
		such that for every $T\geqslant T_2$ 
		\begin{equation}
			\label{eq:auxiliary_inequality_08}
			\sup_{u\in \mathcal{U}}\mathbb{E}\big[\big|L^{u}_{\infty,T}-L^{u}_{\infty}\big|\big]\leqslant \frac{Ce^{-T\beta}}{\beta}<\frac{\epsilon}{3}.
		\end{equation}
		Defining $T_0:=\max\big\{T_1,T_2\big\}$ we write
		\begin{equation*}
			\big|L^{u}_{n,T_0}-L^{u}_{\infty,T_0}\big|\leqslant \int_{0}^{T_0}e^{-\beta s}\big|L\big(X^{u}_{x_n}(s),u(s)\big)-L\big(X^{u}_{x_0}(s),u(s)\big)\big|ds\quad \mathbb{P}\text{-a.s}.
		\end{equation*}
		By hypothesis $\textbf{H}_2$, the operation cost $L$ is Lipschitz continuous in the spatial variable, uniformly with respect to the control variable. Therefore, there exists a constant $C>0$ (independent of $u\in \mathcal{U}$) such that   
		\begin{equation*}
			\big|L^{u}_{n,T_0}-L^{u}_{\infty,T_0}\big|\leqslant C\int_{0}^{T_0}e^{-\beta s}\|X^{u}_{x_n}(s)-X^{u}_{x_0}(s)\|ds\quad \mathbb{P}\text{-a.s}.
		\end{equation*}
		Applying the supremum to the norm that appears in the integrand, we get
		\begin{equation*}
			\big|L^{u}_{n,T_0}-L^{u}_{\infty,T_0}\big|\leqslant \frac{C}{\beta}\bigg(1-e^{-T_0\beta}\bigg)\sup_{r\in [0,T_0]}\|X^{u}_{x_n}(r)-X^{u}_{x_0}(r)\|\quad \mathbb{P}\text{-a.s}.
		\end{equation*}
		Now, applying expectation to both sides of the above inequality, it follows from equation \eqref{eq:norm_S_p_d} for $p=1$ (norm in $\mathcal{S}^1_{\text{d}}\big([0,T_0]\big)$) that
		\begin{equation*}
			\mathbb{E}\big[\big|L^{u}_{n,T_0}-L^{u}_{\infty,T_0}\big|\big]\leqslant \frac{C}{\beta}\bigg(1-e^{-T_0\beta}\bigg)\|X^{u}_{x_n}-X^{u}_{x_0}\|_{\mathcal{S}^1_{\text{d}}\big([0,T_0]\big)}.
		\end{equation*}
		From Lemma \ref{lm:continuity_processes_X_K_VarK}, the application $x\mapsto X^{u}_x(\cdot)$ is equicontinuous with respect to $u$ according to the metric induced by the norm \eqref{eq:norm_space_S_p_d_inf}. Then, since $\lim_{n\to \infty}x_n=x_0$, given $\epsilon>0$ there exists $n_0(T_0)\in\mathbb{N}$ such that for every $n\geqslant n_0(T_0)$ 
		\begin{equation}
			\label{eq:auxiliary_inequality_09}
			\sup_{u\in \mathcal{U}}\mathbb{E}\big[\big|L^{u}_{n,T_0}-L^{u}_{\infty,T_0}\big|\big]\leqslant \frac{\epsilon}{3}.
		\end{equation}
		Therefore, by combining the inequalities \eqref{eq:auxiliary_inequality_07}, \eqref{eq:auxiliary_inequality_08} and \eqref{eq:auxiliary_inequality_09} with the inequality \eqref{eq:auxiliary_inequality_06} we obtain, for all $n\geqslant n_0(T_0)$, that
		\begin{equation*}
			\sup_{u\in \mathcal{U}}\mathbb{E}\big[\big|L^{u}_n-L^{u}_{\infty}\big|\big]\leqslant \epsilon.
		\end{equation*}
		This implies that 
		\begin{equation*}
			\lim_{n\to \infty}\sup_{u\in \mathcal{U}}\big|\mathbb{E}\big[L^{u}_n\big]-\mathbb{E}\big[L^{u}_{\infty}\big]\big|=0.
		\end{equation*}
		\item[Step 3] We want to use \cite[Proposition 1.20]{Pardoux} , to prove that 
		\begin{equation*}
			\lim_{n\to+\infty}\sup_{u\in \mathcal{U}}\big|\mathbb{E}\big[H^{u}_n\big]-\mathbb{E}\big[H^{u}_{\infty}\big]\big|=0.
		\end{equation*}
		Similarly to what we did in the first case, we define the auxiliary sequences $\big(H^{u}_{n,T}\big)_{(n,T)\in\mathbb{N}\times\mathbb{R}_{+}}$ and $\big(H^{u}_{\infty,T}\big)_{T\in\mathbb{R}_{+}}$ respectively by:
		\begin{equation*}
			H^{u}_{n,T}:=\int_{0}^{T}e^{-\beta s}h\big(X^{u}_{x_n}(s)\big)dl^{u}_{x_n}(s)\quad\text{and}\quad
			H^{u}_{\infty,T}:=\int_{0}^{T}e^{-\beta s}h\big(X^{u}_{x_0}(s)\big)dl^{u}_{x_0}(s).
		\end{equation*}
		Based on these sequences, we obtain
		\begin{equation*}
			\big|H^{u}_n-H^{u}_{\infty}\big|\leqslant \big|H^{u}_n-H^{u}_{n,T}\big|+\big|H^{u}_{n,T}-H^{u}_{\infty,T}\big|+\big|H^{u}_{\infty,T}-H^{u}_{\infty}\big|\quad \mathbb{P}\text{-a.s}.
		\end{equation*}
		Let's estimate the first term on the right in the inequality above. By hypothesis \(\textbf{H}_3\), the preventive cost on the boundary is bounded, then there exists a constant $C>0$ such that, $\big|h\big(x\big)\big|\leqslant C$ for all $x\in\partial \mathbb{X}$. Fatou's lemma and integration by parts imply that
		\begin{align*}
			\mathbb{E}\big[\big|H^{u}_n-H^{u}_{n,T}\big|\big]&\leqslant\mathbb{E}\Bigg[\int_{T}^{+\infty}e^{-\beta s}\bigg|h\big(X^{u}_{x_n}(s)\big)\bigg|dl^{u}_{x_n}(s)\Bigg]\leqslant C\mathbb{E}\Bigg[\int_{T}^{+\infty}e^{-\beta s}dl^{u}_{x_n}(s)\Bigg]\\
			&=C\mathbb{E}\Bigg[\liminf_{r\to+\infty}\int_{T}^{r}e^{-\beta s}dl^{u}_{x_n}(s)\Bigg]\leqslant C\liminf_{r\to+\infty}\mathbb{E}\Bigg[\int_{T}^{r}e^{-\beta s}dl^{u}_{x_n}(s)\Bigg]\\
			&=C\liminf_{r\to+\infty}\mathbb{E}\Bigg[\Bigg(l^{u}_{x_n}(r)e^{-\beta r}-l^{u}_{x_n}(T)e^{-\beta T}\Bigg)+\beta\int_{T}^{r}l^{u}_{x_n}(s)e^{-\beta s}ds\Bigg]\\
			&=C\liminf_{r\to+\infty}\Bigg\{\mathbb{E}\Bigg[l^{u}_{x_n}(r)e^{-\beta r}-l^{u}_{x_n}(T)e^{-\beta T}\Bigg]+\beta\mathbb{E}\Bigg[\int_{T}^{r}l^{u}_{x_n}(s)e^{-\beta s}ds\Bigg]\Bigg\}.
		\end{align*}
		Since $-l^{u}_{x_n}(T)e^{-\beta T}\leqslant0$ $\mathbb{P}$-a.s it follows that 
		\begin{equation*}
			\mathbb{E}\big[\big|H^{u}_n-H^{u}_{n,T}\big|\big]\leqslant C\liminf_{r\to+\infty}\Biggl\{e^{-\beta r}\mathbb{E}\big[l^{u}_{x_n}(r)\big]+\beta\mathbb{E}\Bigg[\int_{T}^{r}l^{u}_{x_n}(s)e^{-\beta s}ds\Bigg]\Biggr\}.
		\end{equation*}
		By Fubini's theorem and the inequality \eqref{eq:inequality_l_control_reflection} we obtain that 
		\begin{equation*}
			\mathbb{E}\big[\big|H^{u}_n-H^{u}_{n,T}\big|\big]\leqslant CC_0\liminf_{r\to+\infty}\Biggl\{e^{-\beta r}\big(1+r\big)+\beta\int_{T}^{r}e^{-\beta s}\big(1+s\big)ds\Biggr\}.
		\end{equation*}
		This implies that 
		\begin{equation*}
			\big|\mathbb{E}\big[H^{u}_n\big]-\mathbb{E}\big[H^{u}_{n,T}\big]\big|\leqslant C_0C\beta\int_{T}^{+\infty}e^{-\beta s}\big(1+s\big)ds.
		\end{equation*}
		Now, given $\epsilon>0$ there exists $T_1>0$ such that for every $T\geqslant T_1$
		\begin{equation}
			\label{eq:auxiliary_inequality_10}
			\sup_{u\in \mathcal{U}}\big|\mathbb{E}\big[H^{u}_n\big]-\mathbb{E}\big[H^{u}_{n,T}\big]\big|\leqslant \frac{\epsilon}{3}.
		\end{equation}
		By an analogous argument, there exists $T_2>0$ 
		such that for every $T\geqslant T_2$ 
		\begin{equation}
			\label{eq:auxiliary_inequality_11}
			\sup_{u\in \mathcal{U}}\big|\mathbb{E}\big[H^{u}_{\infty,T}\big]-\mathbb{E}\big[H^{u}_{\infty}\big]\big|\leqslant \frac{\epsilon}{3}.
		\end{equation}
		On the other hand, defining $T_0:=\max\big\{T_1,T_2\big\}$ and given $\eta>0$, we have by Markov's inequality 
		\begin{align*}
			\mathbb{P}\Big(\|X^{u}_{x_n}-X^{u}_{x_0}\|_{T_0}\geqslant \eta\Big)&\leqslant\frac{1}{\eta}\mathbb{E}\big[\|X^{u}_{x_n}-X^{u}_{x_0}\|_{T_0}\big]=\frac{1}{\eta}\|X^{u}_{x_n}-X^{u}_{x_0}\|_{\mathcal{S}^1_{\text{d}}\big([0,T_0]\big)},\\
			\mathbb{P}\Big(\|l^{u}_{x_n}-l^{u}_{x_0}\|_{T_0}\geqslant \eta\Big)&\leqslant\frac{1}{\eta}\mathbb{E}\big[\|l^{u}_{x_n}-l^{u}_{x_0}\|_{T_0}\big]=\frac{1}{\eta}\|l^{u}_{x_n}-l^{u}_{x_0}\|_{\mathcal{S}^1_1\big([0,T_0]\big)}.
		\end{align*}
		Which implies
		\begin{align*}
			\sup_{u\in \mathcal{U}}\mathbb{P}\Big(\|X^{u}_{x_n}-X^{u}_{x_0}\|_{T_0}\geqslant \eta\Big)&\leqslant\frac{1}{\eta}\sup_{u\in \mathcal{U}}\|X^{u}_{x_n}-X^{u}_{x_0}\|_{\mathcal{S}^1_{\text{d}}\big([0,T_0]\big)},\\
			\sup_{u\in \mathcal{U}}\mathbb{P}\Big(\|l^{u}_{x_n}-l^{u}_{x_0}\|_{T_0}\geqslant \eta\Big)&\leqslant\frac{1}{\eta}\sup_{u\in \mathcal{U}}\|l^{u}_{x_n}-l^{u}_{x_0}\|_{\mathcal{S}^1_1\big([0,T_0]\big)}.
		\end{align*}
		Since $\lim_{n\to+\infty}x_n=x_0$, it follows by Lemma \ref{lm:continuity_processes_X_K_VarK} (equicontinuity with respect to u of the applications $x\mapsto X^{u}_x\big(\cdot\big)$ and $x\mapsto l^{u}_x\big(\cdot\big)$), that 
		\begin{equation*}
			\lim_{n\to+\infty}\sup_{u\in \mathcal{U}}\mathbb{P}\Big(\|X^{u}_{x_n}-X^{u}_{x_0}\|_{T_0}\geqslant \eta\Big)=0\quad\text{and}\quad\lim_{n\to+\infty} \sup_{u\in \mathcal{U}}\mathbb{P}\Big(\|l^{u}_{x_n}-l^{u}_{x_0}\|_{T_0}\geqslant \eta\Big)=0.
		\end{equation*}
		Then, the sequences $\big(\|X^{u}_{x_n}-X^{u}_{x_0}\|_{T_0}\big)_{n\geqslant 0}$ and $\big(\|l^{u}_{x_n}-l^{u}_{x_0}\|_{T_0}\big)_{n\geqslant 0}$ converge to zero in probability uniformly in $u\in\mathcal{U}$. This implies that the sequence $\big(\|X^{u}_{x_n}-X^{u}_{x_0}\|_{T_0}+\|l^{u}_{x_n}-l^{u}_{x_0}\|_{T_0}\big)_{n\geqslant 0}$ also converges to zero in probability uniformly in $u\in\mathcal{U}$. 
		By the inequality \eqref{eq:inequality_l_control_reflection} we have, for $p=1$, that 
		\begin{equation*}
			\lambda:=\sup_{n\geqslant0}\mathbb{E}\big[ l^{u}_{x_n}\big(T_0\big)\big]\leqslant C_0\big(1+T_0\big)<\infty.
		\end{equation*}
		Then, by \cite[Proposition 1.20]{Pardoux}, it follows that
		\begin{equation*}
			H^{u}_{n,T_0}=\int_{0}^{T_0}e^{-\beta s} h\big(X^{u}_{x_n}(s)\big)dl^{u}_{x_n}(s)\to \int_{0}^{T_0}e^{-\beta s} h\big(X^{u}_{x_0}(s)\big)dl^{u}_{x_0}(s)=H^{u}_{\infty,T_0}
		\end{equation*}
		in probability uniformly in $u\in\mathcal{U}$. Consequently, the sequence $\big(H^{u}_{n,T_0}\big)_{n\geqslant0}$ converges in distribution to $H^{u}_{\infty,T_0}$ uniformly in $u\in\mathcal{U}$. Since the family $\big\{H^{u}_{n,T_0}\big\}_{n\geqslant0}$ is uniformly integrable, it follows, by \cite[Theorem 1.15]{Pardoux}, that given $\epsilon>0$ there exists $n_0(T_0)\in\mathbb{N}$ such that for every $n\geqslant n_0(T_0)$
		\begin{equation}
			\label{eq:auxiliary_inequality_12}
			\sup_{u\in \mathcal{U}}\big|\mathbb{E}\big[H^{u}_{n,T_0}\big]-\mathbb{E}\big[H^{u}_{\infty,T_0}\big]\big|<\frac{\epsilon}{3}.
		\end{equation}
		Combining the inequalities \eqref{eq:auxiliary_inequality_10}, \eqref{eq:auxiliary_inequality_11} and \eqref{eq:auxiliary_inequality_12}, we get
		\begin{equation*}
			\sup_{u\in \mathcal{U}}\big|\mathbb{E}\big[H^{u}_n\big]-\mathbb{E}\big[H^{u}_{\infty}\big]\big|<\epsilon.
		\end{equation*}
		\item[Step 4](Conclusion) Finally, we get
		\begin{equation*}
			\lim_{n\to+\infty}\sup_{u\in \mathcal{U}}\big|J^{\beta}_{x_n}\big(u\big)-J^{\beta}_{x_0}\big(u\big)\big|\leqslant \lim_{n\to+\infty}\sup_{u\in \mathcal{U}}\big|\mathbb{E}\big[L^{u}_n\big]- \mathbb{E}\big[L^{u}_{\infty}\big]\big|+\lim_{n\to+\infty}\sup_{u\in \mathcal{U}}\big|\mathbb{E}\big[H^{u}_n\big]-\mathbb{E}\big[H^{u}_{\infty}\big]\big|=0.
		\end{equation*}
		We conclude that the family $\big(x\mapsto J^{\beta}_{x}\big(u\big)\big)_{u\in\mathcal{U}}$ is equicontinuous.
	\end{itemize}
\end{proof}
\begin{theorem}[Continuity of the Optimal Value Function]
	\label{tm:continuity_optimal_value_function}
	The Optimal Value Function \eqref{eq:optimal_value_function} is continuous.
\end{theorem}
\begin{proof}
	On the one hand, given $x_0,x\in\mathbb{X}$, we have that
	\begin{align*}
		-J^{\beta}_x\big(\widetilde{u}\big)&=J^{\beta}_{x_0}\big(\widetilde{u}\big)-J^{\beta}_x\big(\widetilde{u}\big)-J^{\beta}_{x_0}\big(\widetilde{u}\big)\leqslant\big|J^{\beta}_{x_0}\big(\widetilde{u}\big)-J^{\beta}_x\big(\widetilde{u}\big)\big|-J^{\beta}_{x_0}\big(\widetilde{u}\big)\\
		&\leqslant\sup_{u\in \mathcal{U}}\big|J^{\beta}_{x_0}\big(u\big)-J^{\beta}_x\big(u\big)\big|+\sup_{u\in \mathcal{U}}\biggl\{-J^{\beta}_{x_0}\big(u\big)\biggr\}\\
		&=\sup_{u\in \mathcal{U}}\big|J^{\beta}_{x_0}\big(u\big)-J^{\beta}_x\big(u\big)\big|-\inf_{u\in \mathcal{U}}J^{\beta}_{x_0}\big(u\big).
	\end{align*}
	The above inequality implies 
	\begin{align*}
		-\inf_{u\in\mathcal{U}}J^{\beta}_x\big(u\big)=\sup_{u\in \mathcal{U}}\biggl\{-J^{\beta}_x\big(u\big)\biggr\}\leqslant \sup_{u\in \mathcal{U}}\big|J^{\beta}_{x_0}\big(u\big)-J^{\beta}_x\big(u\big)\big|-\inf_{u\in \mathcal{U}}J^{\beta}_{x_0}\big(u\big).
	\end{align*}
	In other words,
	\begin{equation*}
		\inf_{u\in\mathcal{U}}J^{\beta}_{x_0}\big(u\big)-\inf_{u\in \mathcal{U}}J^{\beta}_x\big(u\big)\leqslant\sup_{u\in \mathcal{U}}\big|J^{\beta}_{x_0}\big(u\big)-J^{\beta}_x\big(u\big)\big|.
	\end{equation*}
	On the other hand, since the argument is symmetrical, it follows that
	\begin{equation*}
		\inf_{u\in \mathcal{U}}J^{\beta}_x\big(u\big)-\inf_{u\in \mathcal{U}}J^{\beta}_{x_0}\big(u\big)\leqslant  \sup_{u\in \mathcal{U}}\big|J^{\beta}_{x_0}\big(u\big)-J^{\beta}_x\big(u\big)\big|.
	\end{equation*}
	Thus, we get
	\begin{equation*}
		\big|v^{\beta}\big(x_0\big)-v^{\beta}\big(x\big)\big|=\Bigg|\inf_{u\in\mathcal{U}}J^{\beta}_{x_0}\big(u\big)-\inf_{u\in \mathcal{U}}J^{\beta}_x\big(u\big)\Bigg|\leqslant  \sup_{u\in \mathcal{U}}\big|J^{\beta}_{x_0}\big(u\big)-J^{\beta}_x\big(u\big)\big|.
	\end{equation*}
	By Lemma \ref{lm:continuity_functional_cost_initial_condition} given $\epsilon>0$ there exists $\delta_{x_0}(\epsilon)>0$ such that
	\begin{equation*}
		\|x_0-x\|\leqslant \delta_{x_0}(\epsilon)\Rightarrow \sup_{u\in \mathcal{U}}\big|J^{\beta}_{x_0}\big(u\big)-J^{\beta}_x\big(u\big)\big|<\epsilon.
	\end{equation*}
	This implies
	\begin{equation*}
		\|x_0-x\|\leqslant \delta_{x_0}(\epsilon)\Rightarrow \big|v^{\beta}\big(x_0\big)-v^{\beta}\big(x\big)\big|<\epsilon.
	\end{equation*}
\end{proof}
\subsection{The Dynamic Programming Principle}
\label{subsec:dynamic_programming_principle}
Dynamic programming is a widely used methodology for solving problems involving multi-stage or sequential decision-making in discrete time; see, for example, \cite[Chapter 12]{Parmigiani} and \cite[Chapter 4]{Kochenderfer}. This methodology relies on an optimality principle known as the Dynamic Programming Principle (DPP). For a thorough theoretical development, see \cite{Bertsekas01}; for a discussion on the practical challenges and algorithmic implementations, refer to \cite{Powell}.

The DPP can also be extended to continuous-time decision problems, such as those arising in optimal control theory. In the deterministic setting, its mathematical formulation is relatively straightforward; see, for instance, \cite[Chapter 12,  Lemma 202]{Baumeister}. However, in the stochastic setting, the DPP becomes a highly technical result. The level of complexity depends on the nature of the control strategies considered. For example, if feedback-type controls are used, then strong solutions to the controlled SDE \eqref{eq:sde_reflection_control} may exist only in very specific situations.

Before presenting the main result, we make a few additional remarks. The admissible controls under consideration are adapted to the filtration $\big(\mathcal{F}_t\big)_{t\geq 0}$, which is the completion of the filtration generated by the Brownian motion. Consequently, these controls can be represented as measurable functions of the Brownian motion’s trajectory. This leads, by arguments similar to those in \cite[Chapter 8, Theorem 8.5]{LeGall}, to the existence of measurable functions $\mathfrak{X}$ and $\mathfrak{L}$ such that
\begin{equation*}
	X^{u}_x(t)=\mathfrak{X}^{u}_x\big(W\big)(t)\quad\text{and}\quad l^{u}_x(t)=\mathfrak{L}^{u}_x\big(W\big)(t). 
\end{equation*} 
Thus, given a stopping time $\tau:\Omega\to \mathbb{R}_{+}$ with respect to $\big(\mathcal{F}_t\big)_{t\geq0}$ and the $\sigma$-algebra 
\begin{equation*}
	\mathcal{F}_{\tau}:=\Bigg\{A\in \mathcal{O}:\big\{\tau\leq t\big\}\cap A\in \mathcal{F}_t,\hspace{0.15cm}\forall t\in\mathbb{R}_{+}\Bigg\},
\end{equation*}
the following equality holds 
\begin{equation*}
	\mathbb{E}\Bigg[\int_{\tau}^{+\infty}e^{-\beta (s-\tau)}L\big(X_{X_x(\tau)}(s), u(s)\big)\,ds+\int_{\tau}^{+\infty}e^{-\beta (s-\tau)}h\big(X_{X_x(\tau)}(s)\big)\,dl_{X_x(\tau)}(s)\Bigg|\mathcal{F}_{\tau}\Bigg]=J^{\beta}_{X_x(\tau)}\big(\widetilde{u}\big)\quad \mathbb{P}\text{-a.s}
\end{equation*}
where $\widetilde{u}$ is the restriction of the control $u$ to the interval $[\tau,+\infty)$. The above result can also be obtained, by a convenient adaptation of Lemma 3.2 from \cite[Chapter 4]{Zhou}. 

For the case with an infinite horizon and reflection on the boundary we have.
\begin{theorem}[Dynamic Programming Principle]
	\label{tm:dynamic_programming_principle}
	Consider hypotheses $\emph{\textbf{H}}_1$, $\emph{\textbf{H}}_2$ and $\emph{\textbf{H}}_3$. Then, for all $x\in \mathbb{X}:$
	\begin{subequations}
		\label{subeq:dynamic_programming_principle}
		\begin{align}
			\label{eq:dynamic_programming_principle_inf_inf}
			v^{\beta}\big(x\big)=\inf_{u\in \mathcal{U}}\inf_{\tau\in \mathcal{T}}\mathbb{E}\Bigg[&\int_{0}^{\tau}e^{-\beta s}L\big(X_x(s), u(s)\big)\,ds+\int_{0}^{\tau}e^{-\beta s}h\big(X_x(s)\big)\,dl_x(s)+e^{-\beta\tau}v^{\beta}\big(X_x(\tau)\big)\Bigg] \\
			\label{eq:dynamic_programming_principle_inf_sup}
			=\inf_{u\in \mathcal{U}}\sup_{\tau\in \mathcal{T}}\mathbb{E}\Bigg[&\int_{0}^{\tau}e^{-\beta s}L\big(X_x(s), u(s)\big)\,ds+\int_{0}^{\tau}e^{-\beta s}h\big(X_x(s)\big)\,dl_x(s) +e^{-\beta\tau}v^{\beta}\big(X_x(\tau)\big)\Bigg],
		\end{align}
	\end{subequations}
	where $\mathcal{T}$ is the set of stopping times with respect to $\big(\mathcal{F}_t\big)_{t\geq0}$. In addition, we adopt the convention that $e^{-\beta \tau(\omega)}=0$ for all $\omega\in\Omega$ such that $\tau(\omega)=+\infty$.
\end{theorem}
\begin{proof}  
	According to hypotheses $\textbf{H}_1$, $\textbf{H}_2$ and $\textbf{H}_3$, the three processes that appear in equations \eqref{eq:dynamic_programming_principle_inf_inf} and \eqref{eq:dynamic_programming_principle_inf_sup} are measurable and bounded, so the expectation is well defined. Given $x\in\mathbb{X}$ and $u\in\mathcal{U}$, it follows from the conditional expectation property that
	\begin{align*}
		J^{\beta}_x\big(u\big)=\mathbb{E}\Bigg[&\int_{0}^{+\infty}e^{-\beta s}L\big(X_x(s),u(s)\big)\,ds +\int_{0}^{+\infty}e^{-\beta s}h\big(X_x(s)\big)\,dl_x(s)\Bigg]\\
		=\mathbb{E}\Bigg[&\int_{0}^{\tau}e^{-\beta s}L\big(X_x(s),u(s)\big)\,ds+\int_{0}^{\tau}e^{-\beta s}h\big(X_x(s)\big)\,dl_x(s)\\ 
		&+e^{-\beta \tau}\mathbb{E}\Bigg[\int_{\tau}^{+\infty}e^{-\beta (s-\tau)}L\big(X_x(s),u(s)\big)\,ds+\int_{\tau}^{+\infty}e^{-\beta (s-\tau)}h\big(X_x(s)\big)\,dl_x(s)\Bigg|\mathcal{F}_{\tau}\Bigg]\Bigg]\\
		=\mathbb{E}\Bigg[&\int_{0}^{\tau}e^{-\beta s}L\big(X_x(s),u(s)\big)\,ds+\int_{0}^{\tau}e^{-\beta s}h\big(X_x(s)\big)\,dl_x(s)\\ 
		&+e^{-\beta \tau}\mathbb{E}\Bigg[\int_{\tau}^{+\infty}e^{-\beta (s-\tau)}L\big(X_{X_x(\tau)}(s),u(s)\big)\,ds+\int_{\tau}^{+\infty}e^{-\beta (s-\tau)}h\big(X_{X_x(\tau)}(s)\big)\,dl_{X_x(\tau)}(s)\Bigg|\mathcal{F}_{\tau}\Bigg]\Bigg]\\
		=\mathbb{E}\Bigg[&\int_{0}^{\tau}e^{-\beta s}L\big(X_x(s),u(s)\big)\,ds+\int_{0}^{\tau}e^{-\beta s}h\big(X_x(s)\big)\,dl_x(s) +e^{-\beta \tau}J^{\beta}_{X_x(\tau)}\big(\widetilde{u}\big)\Bigg]
	\end{align*}
	where $\widetilde{u}$ is the restriction of the control $u$ to the interval $[\tau,+\infty)$. Since $v^{\beta}\big(X_x(\tau)\big)\leqslant J^{\beta}_{X_x(\tau)}\big(\widetilde{u}\big)$ $\mathbb{P}$-a.s, we have that 
	\begin{equation*}
		J^{\beta}_x\big(u\big)\geqslant\mathbb{E}\Bigg[\int_{0}^{\tau}e^{-\beta s}L\big(X_x(s),u(s)\big)\,ds +\int_{0}^{\tau}e^{-\beta s}h\big(X_x(s)\big)\,dl_x(s)+e^{-\beta \tau}v^{\beta}\big(X_x(\tau)\big)\Bigg].
	\end{equation*}
	Given that $\tau\in \mathcal{T}$ is arbitrary 
	\begin{equation*}
		J^{\beta}_{x}\big(u\big)\geqslant\sup_{\tau\in \mathcal{T}}\mathbb{E}\Bigg[\int_{0}^{\tau}e^{-\beta s}L\big(X_x(s),u(s)\big)\,ds +\int_{0}^{\tau}e^{-\beta s}h\big(X_x(s)\big)\,dl_x(s)+e^{-\beta \tau}v^{\beta}\big(X_x(\tau)\big)\Bigg].
	\end{equation*}
	Taking the infimum on both sides with respect to the controls, we get 
	\begin{equation*}
		v^{\beta}\big(x\big)\geqslant\inf_{u\in \mathcal{U}}\sup_{\tau\in \mathcal{T}}\mathbb{E}\Bigg[\int_{0}^{\tau}e^{-\beta s}L\big(X_x(s),u(s)\big)\,ds+\int_{0}^{\tau}e^{-\beta s}h\big(X_x(s)\big)\,dl_x(s) +e^{-\beta \tau}v^{\beta}\big(X_x(\tau)\big)\Bigg].
	\end{equation*}    	
	
	On the other hand, given a control $u\in\mathcal{U}$ and a stopping time $\tau\in\mathcal{T}$, we have by the definition of the optimal value function \eqref{eq:optimal_value_function} that given $\rho>0$ there exists $\widetilde{u}^{\rho}\in\mathcal{U}$ such that
	\begin{equation}
		\label{eq:auxiliary_equation_01}
		v^{\beta}\big(X^{u}_x(\tau)\big)+\rho\geqslant J^{\beta}_{X^{u}_x(\tau)}\big(\widetilde{u}^{\rho}\big)\quad \mathbb{P}\text{-a.s}.
	\end{equation} 
	Define the control $\big(\hat{u}(t)\big)_{t\geqslant0}$ by
	\begin{equation}
		\label{eq:auxiliary_equation_02}
		\hat{u}(t):=u(t)\mathds{1}_{[0,\tau]}(t)+\widetilde{u}^{\rho}(t)\mathds{1}_{(\tau,+\infty)}(t).
	\end{equation}
   The measurability of the control defined in \eqref{eq:auxiliary_equation_02} is a subtle and nontrivial matter. A detailed treatment of this issue can be found in the literature; see, for instance, \cite[Chapter 3]{Pham}. However, by applying a measurable selection theorem --- such as those presented in \cite[Chapter 4]{Soner} or, more directly, in \cite[Appendix B]{Rishel} --- one can rigorously establish that the control in question admits a progressively measurable version. As a result, it belongs to the class $\mathcal{U}$.
	
    Therefore, by standard properties of the conditional expectation, it follows that
	\begin{equation*}
		v^{\beta}\big(x\big)\leqslant J^{\beta}_x\big(\hat{u}\big)=\mathbb{E}\Bigg[\int_{0}^{\tau}e^{-\beta s}L\big(X_x(s),u(s)\big)\,ds +\int_{0}^{\tau}e^{-\beta s}h\big(X_x(s)\big)\,dl_x(s)+e^{-\beta \tau}J^{\beta}_{X_x(\tau)}\big(\widetilde{u}^{\rho}\big)\Bigg].
	\end{equation*}
	From \eqref{eq:auxiliary_equation_01}, it follows that
	\begin{equation*}
		v^{\beta}\big(x\big)\leqslant\mathbb{E}\Bigg[\int_{0}^{\tau}e^{-\beta s}L\big(X_x(s),u(s)\big)\,ds
		+\int_{0}^{\tau}e^{-\beta s}h\big(X_x(s)\big)\,dl_x(s)+e^{-\beta \tau}v^{\beta}\big(X_x(\tau)\big)\Bigg]+\rho.
	\end{equation*}
	Taking the infimum in relation to $\tau$ and $u$ (since both are arbitrary), we get
	\begin{equation*}
		v^{\beta}\big(x\big)\leqslant\inf_{u\in \mathcal{U}}\inf_{\tau\in \mathcal{T}}\mathbb{E}\Bigg[\int_{0}^{\tau}e^{-\beta s}L\big(X_x(s), u(s)\big)\,ds+\int_{0}^{\tau}e^{-\beta s}h\big(X_x(s)\big)\,dl_x(s)+e^{-\beta \tau}v^{\beta}\big(X_x(\tau)\big)\Bigg]+\rho.
	\end{equation*}
	Since $\rho>0$ is arbitrary
	\begin{equation*}
		v^{\beta}\big(x\big)\leqslant\inf_{u\in \mathcal{U}}\inf_{\tau\in \mathcal{T}}\mathbb{E}\Bigg[\int_{0}^{\tau}e^{-\beta s}L\big(X_x(s), u(s)\big)\,ds+\int_{0}^{\tau}e^{-\beta s}h\big(X_x(s)\big)\,dl_x(s)+e^{-\beta \tau}v^{\beta}\big(X_x(\tau)\big)\Bigg].
	\end{equation*} 
	Then the following inequalities hold:
	\begin{align*}
		v^{\beta}\big(x\big)\leqslant&\inf_{u\in \mathcal{U}}\inf_{\tau\in \mathcal{T}}\mathbb{E}\Bigg[\int_{0}^{\tau}e^{-\beta s}L\big(X_x(s), u(s)\big)ds+\int_{0}^{\tau}e^{-\beta s}h\big(X_x(s)\big)dl_x(s)+e^{-\beta \tau}v^{\beta}\big(X_x(\tau)\big)\Bigg]\\
		\leqslant&\inf_{u\in \mathcal{U}}\sup_{\tau\in \mathcal{T}}\mathbb{E}\Bigg[\int_{0}^{\tau}e^{-\beta s}L\big(X_x(s), u(s)\big)ds+\int_{0}^{\tau}e^{-\beta s}h\big(X_x(s)\big)dl_x(s)+e^{-\beta \tau}v^{\beta}\big(X_x(\tau)\big)\Bigg]\leqslant v^{\beta}\big(x\big).
	\end{align*}
\end{proof}
\subsection{Viscosity Solutions for the Hamilton--Jacobi--Bellman Equation with Neumann Boundary Condition}
\label{subsec:viscosity_solution_HJB_equation_von_Neumann_boundary_condition}
The HJB equation associated with the control problem under consideration is a fully nonlinear, second-order Partial Differential Equation (PDE). It can be formally interpreted as the infinitesimal counterpart of the DPP \eqref{subeq:dynamic_programming_principle}; see, for example, \cite[Chapter 5, Theorem 3.1]{Soner}. To derive this equation from the DPP, one typically applies Itô’s formula to the optimal value function \eqref{eq:optimal_value_function}, which requires, as an additional assumption, that $v^\beta \in \text{C}^2\big(\mathbb{X}\big)$. Moreover, the derivation presumes the existence of optimal controls of Markovian type. Hence, this derivation should be regarded as heuristic.

The HJB equation corresponding to the control problem is given by
\begin{equation}
	\label{eq:pde_hamilton_jacobi_bellman}
	\beta v^{\beta}\big(x\big) - \mathcal{H}\Big(x, D_x v^{\beta}\big(x\big), D^2_x v^{\beta}\big(x\big) \Big) = 0, \quad \forall x \in \mathbb{X},
\end{equation}
subject to the following Neumann-type boundary condition:
\begin{equation}
	\label{eq:von_Neumann_boundary_condition}
	\big\langle D_xv^{\beta}\big(x\big), D_x\phi\big(x\big)\big\rangle=h\big(x\big)\quad \forall x\in \partial \mathbb{X}.
\end{equation}
where the Hamiltonian $\mathcal{H}: \mathbb{X}\times\mathbb{R}^{d}\times \mathbb{S}^{d}\big(\mathbb{R}\big)\to\mathbb{R}$ is thus given 
\begin{equation}
	\label{eq:Hamiltonian}
	\mathcal{H}\Big(x,g_x,H_x\Big):=\inf_{u\in \mathbb{U}}\Bigg\{\Big\langle\mu_{X}\big(x,u\big),g_x\Big\rangle +\frac{1}{2}\text{Tr}\Big(H_x\big[\sigma_{X}\sigma^{\top}_{X}\big]\big(x,u\big)\Big)+L\big(x,u\big)\Bigg\}
\end{equation}
and $\mathbb{S}^d\big(\mathbb{R}\big)$ denotes the space of symmetric $d \times d$ real matrices.

Due to the absence of structural assumptions on the diffusion matrix $\sigma_X$ in the controlled SDE \eqref{eq:sde_reflection_control}, the resulting HJB equation is in general degenerate elliptic. This degeneracy excludes the existence of classical solutions or even strong (Sobolev) solutions in many cases. The presence of the boundary condition \eqref{eq:von_Neumann_boundary_condition} further complicates the analysis and hinders the applicability of standard PDE techniques.

For this reason, until the late 1970s, most research on SOC focused on problems with uniformly elliptic, control-independent diffusion matrices, for which the associated HJB equations are semilinear and uniformly elliptic. In this setting, the existence of classical solutions can often be established; see, for instance, \cite[Theorem 4.4.3]{Ghosh}. Around the same time, progress was made on fully nonlinear, uniformly elliptic HJB equations under stronger regularity assumptions, using the notion of strong (Sobolev) solutions; see Theorems 1 and 2 (along with their hypotheses) in \cite[Chapter 4]{Krylov}.

A major advance came in the early 1980s with the introduction of the theory of viscosity solutions by Crandall, Ishii and Lions; see \cite{Crandall}. This framework, which applies to fully nonlinear, possibly degenerate elliptic (and parabolic) equations, only requires continuity of the solution and thus allows for a rigorous connection between SOC problems --- whose value functions are typically only continuous --- and their associated HJB equations. For historical context, see,  \cite[Chapter 4]{Zhou}.

We now reformulate the PDE system \eqref{eq:pde_hamilton_jacobi_bellman}–\eqref{eq:von_Neumann_boundary_condition} in a form suitable for the application of viscosity solution theory. Accordingly, throughout the remainder of this section, we focus on the following second-order, nonlinear PDE:
\begin{subequations}
	\label{subeq:pde_hjb_and_von_Neumann_boundary_condition}
	\begin{align}
		\label{eq:pde_hjb}
		F^{\beta}\Big(x,v^{\beta}\big(x\big),D_{x}v^{\beta}\big(x\big),D^2_{x}v^{\beta}\big(x\big)\Big)&=0\quad\forall x\in \mathbb{X},\\
		\label{eq:pde_hjb_von_Neumann_boundary_condition}
		\Gamma\Big(x,D_xv^{\beta}\big(x\big)\Big)&=0\quad\forall x\in\partial \mathbb{X},
	\end{align}
\end{subequations}
where $F^{\beta}:\mathbb{X}\times\mathbb{R}\times\mathbb{R}^{d}\times\mathbb{S}^{d}\big(\mathbb{R}\big)\to\mathbb{R}$ is defined by
\begin{equation*}
	F^{\beta}\Big(x,r,g_x,H_x\Big):=\beta r-\mathcal{H}\Big(x,g_x,H_x\Big),
\end{equation*}
with Hamiltonian $\mathcal{H}$ defined in \eqref{eq:Hamiltonian}, $h:\partial \mathbb{X}\to\mathbb{R}$ satisfying hypothesis $\textbf{H}_3$, $\phi\in \text{C}^2_{\text{b}}\big(\mathbb{R}^d\big)$ and the set $\mathbb{X}$ satisfying the conditions listed in the introduction of the Section \S\ref{sec:stochastic_optimal_control_restrictions_state_variables}. In addition, we define $\Gamma:\partial \mathbb{X}\times\mathbb{R}^d\to\mathbb{R}$ to represent the boundary condition \eqref{eq:von_Neumann_boundary_condition} 
\begin{equation*}
	\Gamma\big(x,g_x\big):=\Big\langle g_x,D_x\phi\big(x\big)\Big\rangle-h\big(x\big).
\end{equation*}
The definition of viscosity solutions adopted in this work is a slight modification of the one presented in \cite[Chapter 16]{Kushner}, where we use test functions of class $\text{C}^{\infty}\big(\mathbb{X}\big)$ instead of test functions of class $\text{C}^2\big(\mathbb{X}\big)$.
\begin{remark}
	$\text{USC}\big(\mathbb{X}\big)$ denotes the set of upper semi-continuous functions, and $\text{LSC}\big(\mathbb{X}\big)$ denotes the set of lower semi-continuous functions.
\end{remark}
\begin{definition}[Viscosity Solutions]
	\label{df:viscosity_solutions}
	\text{ }
	\begin{itemize}
		\item $w\in \emph{USC}\big(\mathbb{X}\big)$ is an \emph{viscosity subsolution}, if for every $\varphi\in \emph{C}^{\infty}\big(\mathbb{X}\big)$ such that, the application $x\mapsto\big(w-\varphi\big)\big(x\big)$ has a local maximum at the point $x\in \mathbb{X}$ with $w\big(x\big)=\varphi\big(x\big)$, the following inequalities hold:
		\begin{subequations}
			\label{subeq:viscosity_subsolution}
			\begin{align}
				\label{eq:inner_boundary_viscosity_subsolution}
				F^{\beta}\Big(x,\varphi\big(x\big),D_{x}\varphi\big(x\big),D^2_{x}\varphi\big(x\big)\Big)&\leqslant0\quad x\in \mathbb{X},\\
				\label{eq:boundary_viscosity_subsolution}
				F^{\beta}\Big(x,\varphi\big(x\big),D_x\varphi\big(x\big), D^2_x\varphi\big(x\big)\Big)\wedge\Gamma\Big(x,D_x\varphi\big(x\big)\Big)&\leqslant0\quad x\in\partial \mathbb{X}.
			\end{align}
		\end{subequations}	
		\item $w\in \emph{LSC}\big(\mathbb{X}\big)$ is an \emph{viscosity supersolution}, if for every $\varphi\in \emph{C}^{\infty}\big(\mathbb{X}\big)$ such that, the application $x\mapsto\big(w-\varphi\big)\big(x\big)$ has a local minimum at the point $x\in \mathbb{X}$ with $w\big(x\big)=\varphi\big(x\big)$, the following inequalities hold:
		\begin{subequations}
			\label{subeq:viscosity_supersolution}
			\begin{align}
				\label{eq:inner_boundary_viscosity_supersolution}
				F^{\beta}\Big(x,\varphi\big(x\big),D_x\varphi\big(x\big), D^2_x\varphi\big(x\big)\Big)&\geqslant0\quad x\in \mathbb{X},\\
				\label{eq:boundary_viscosity_supersolution}
				F^{\beta}\Big(x,\varphi\big(x\big),D_x\varphi\big(x\big),
				D^2_x\varphi\big(x\big)\Big)\lor\Gamma\Big(x,D_x\varphi\big(x\big)\Big)&\geqslant0\quad x\in\partial \mathbb{X}.
			\end{align}
		\end{subequations}
		\item $w\in \emph{C}\big(\mathbb{X}\big)$ is an \emph{viscosity solution} if it is at the same time a subsolution and a supersolution. 
	\end{itemize}
\end{definition}
 We aim to show that the optimal value function \eqref{eq:optimal_value_function} is a viscosity solution of the HJB equation \eqref{eq:pde_hamilton_jacobi_bellman}. The proof is carried out in two steps. First, we establish that the optimal value function is a viscosity subsolution. To this end, we introduce an auxiliary lemma that ensures the existence of Lipschitz continuous Markovian controls under suitable conditions.

In the second step, we prove that the optimal value function is also a viscosity supersolution. Although the overall structure of the argument --- based on a contradiction --- is similar to the subsolution case, additional technical considerations are required in this part of the proof.
\begin{lemma}
	\label{lm:Lipschitz_controls} 
	Consider a set $\mathbb{U}\subset \mathbb{R}^m$ and a function $\theta:\mathbb{X}\times \mathbb{U}\to\mathbb{R}$ satisfying the following hypotheses:
	\begin{itemize}
		\item $\mathbb{U}$ is a convex and compact set;
		\item $\theta\big(x,\cdot\big)\in \emph{C}^2\big(\mathbb{U}\big)$ for all $x\in \mathbb{X}$;
		\item $D_u\theta\big(\cdot,u\big)$ is Lipschitz continuous uniformly in $u\in \mathbb{U}$;
		\item $D^2_u\theta$ is positive definite uniformly on $(x,u)\in\mathbb{X}\times \mathbb{U}$.
	\end{itemize}
	Then the function 
	\begin{equation}
		\label{eq:measurable_function_markovian_optimal_control}
		u^{*}\big(x\big)\in \arg\min_{u\in \mathbb{U}}\theta\big(x,u\big)
	\end{equation}
	is Lipschitz continuous.
\end{lemma}
\begin{proof}
	From item 4, $\theta\big(x,\cdot\big)$ is a strictly convex function uniformly on $x\in\mathbb{X}$. This fact, together with item 1, guarantees that there is only one function $u^{*}$ satisfying the problem \eqref{eq:measurable_function_markovian_optimal_control}.  
	Denote $u_1:=u^{*}\big(x_1\big)$ and $u_2:=u^{*}\big(x_2\big)$. From item 2, it follows from Taylor expansion with Lagrange remainder that there exists a point $(\widetilde{x},\widetilde{u})\in \mathbb{X}\times \mathbb{U}$ such that
	\begin{equation*}
		\theta\big(x_1,u_2\big)=\theta\big(x_1,u_1+\big(u_2-u_1\big)\big)=\theta\big(x_1,u_1\big)+\bigg\langle D_u\theta\big(x_1,u_1\big),\big(u_2-u_1\big)\bigg\rangle+\frac{1}{2}\bigg\langle \big(u_2-u_1\big),D^2_u\theta\big(\widetilde{x},\widetilde{u}\big)\big(u_2-u_1\big)\bigg\rangle. 
	\end{equation*}
	From item 4 there is a constant $\varkappa>0$ such that 
	\begin{equation*}
		\theta\big(x_1,u_2\big)\geqslant\theta\big(x_1,u_1\big)+\bigg\langle D_u\theta\big(x_1,u_1\big),\big(u_2-u_1\big)\bigg\rangle+\frac{\varkappa}{2}\|u_2-u_1\|^2. 
	\end{equation*}
	Since $u_1$ is the minimum point for the application $u\mapsto\theta\big(x_1,\cdot\big)$, it follows that 
	\begin{equation*}
		\theta\big(x_1,u_2\big)-\theta\big(x_1,u_1\big)\geqslant\frac{\varkappa}{2}\|u_2-u_1\|^2. 
	\end{equation*}
	By the convexity of $\mathbb{U}$, the line $[0,1]\ni s\mapsto u_1+s\big(u_2-u_1\big)\in \mathbb{U}$. Let $\mathfrak{P}:\mathbb{X}\times [0,1]\to \mathbb{X}\times \mathbb{U}$ be defined by
	\begin{equation*}
		\mathfrak{P}\big(x,s\big):=\big(x,u_1+s\big(u_2-u_1\big)\big).
	\end{equation*}
	Then, 
	\begin{equation*}
		\int_{0}^{1}\bigg\langle D_u\theta\big(\mathfrak{P}\big(x_1,s\big)\big), u_2-u_1\bigg\rangle ds\geqslant\frac{\varkappa}{2}\|u_2-u_1\|^2.
	\end{equation*}
	Similarly, by changing the roles of $x_1$ and $x_2$, we get that 
	\begin{equation*}
		-\int_{0}^{1}\bigg\langle D_u\theta\big(\mathfrak{P}\big(x_2,s\big)\big), u_2-u_1\bigg\rangle ds\geqslant\frac{\varkappa}{2}\|u_2-u_1\|^2.
	\end{equation*}
	Adding up the inequalities above, we get 
	\begin{equation*}
		\int_{0}^{1}\bigg\langle D_u\theta\big(\mathfrak{P}\big(x_1,s\big)\big)-D_u\theta\big(\mathfrak{P}\big(x_2,s\big)\big), u_2-u_1\bigg\rangle ds\geqslant\varkappa\|u_2-u_1\|^2.
	\end{equation*}
	From the Cauchy-Schwarz inequality, it follows that 
	\begin{equation*}
		\int_{0}^{1}\| D_u\theta\big(\mathfrak{P}\big(x_1,s\big)\big)-D_u\theta\big(\mathfrak{P}\big(x_2,s\big)\big)\|ds\geqslant\varkappa\|u_2-u_1\|.
	\end{equation*}
	By item 3, $D_u\theta\big(\cdot,u\big)$ is Lipschitz continuous uniformly in $u\in \mathbb{U}$. Therefore, there exists a constant $C_0>0$ (independent of $u\in \mathbb{U}$) such that 
	\begin{equation*}
		\|D_u\theta\big(\mathfrak{P}\big(x_1,s\big)\big)-D_u\theta\big(\mathfrak{P}\big(x_2,s\big)\big)\|\leqslant C_0\|x_1-x_2\|\quad \forall s\in [0,1].
	\end{equation*}
	Combining the above inequality with the previous one, we obtain that 
	\begin{equation*}
		\|u^{*}\big(x_2\big)-u^{*}\big(x_1\big)\|\leqslant \frac{C_0}{\varkappa}\|x_2-x_1\|.
	\end{equation*}
\end{proof}
\begin{remark}
	\label{rm:remark_01}
	Given $\varphi \in \text{C}^{\infty}\big(\mathbb{X}\big)$, consider the function $\theta_{\varphi}:\mathbb{X}\times\mathbb{U}\to\mathbb{R}$ defined by
	\begin{equation}
		\label{eq:function_theta_varphi}
		\theta_{\varphi}\big(x, u\big) := \mathcal{L}^{u}_X \varphi\big(x\big) + L\big(x, u\big),
	\end{equation}
	where $\mathcal{L}^{u}_X$ denotes the second-order differential operator introduced in \eqref{eq:auxiliary_equation_08}.
	
	A concrete setting in which Hypotheses 1–4 of Lemma \ref{lm:Lipschitz_controls} are satisfied for the function $\theta_{\varphi}$ defined above occurs under the following structural conditions:
	\begin{itemize}
		\item[5.1] $\mu_X\big(x,\cdot\big)$, $\sigma_X\big(x,\cdot\big)$ and $L\big(x,\cdot\big)$ are of class $\text{C}^{2}\big(\mathbb{U}\big)$ for all $x\in\mathbb{X}$. In addition, the functions $D_{xu}\mu_X$ and $D_{xu}\sigma_X$ are of class $\text{C}\big(\mathbb{X}\times\mathbb{U}\big)$. 
		\item[5.2] $D^2_u\mu_X\equiv 0, \quad D^2_u\big[\sigma_{X}\sigma^{\top}_{X}\big]\equiv 0 \quad \text{and}\quad D^2_uL\succ 0$.
	\end{itemize}
	
	From Hypothesis 5.2, the hessian of the function \eqref{eq:function_theta_varphi} is given by $D^2_u\theta_{\varphi}=D^2_uL$ and is positive definite as in item 4. From hypothesis 5.1, it follows from the mean value theorem that 
	\begin{equation*}
		\|D_u\mu_X\big(x_1,u\big)-D_u\mu_X\big(x_2,u\big)\|\leqslant \|D_{xu}\mu_X\big(\tilde{x},\tilde{u}\big)\|\|x_2-x_1\|,
	\end{equation*}
	where $(\widetilde{x},\widetilde{u})$ is a point on the line connecting the points $(x_1,u)$ and $(x_2,u)$. Since $D_{xu}\mu_X$ is continuous and $\mathbb{X}\times \mathbb{U}$ is compact, there is a constant $C>0$ such that 
	\begin{equation*}
		\|D_u\mu_X\big(x_1,u\big)-D_u\mu_X\big(x_2,u\big)\|\leqslant C\|x_2-x_1\|
	\end{equation*}
	with the same argument, valid for $D_u\big[\sigma_{X}\sigma^{\top}_{X}\big]$ and $D_uL$. Since these functions are Lipschitz and bounded, and $D_x\varphi$ and $D^2_x\varphi$ are also Lipschitz and bounded, it follows that 
	\begin{equation*}
		\|D_u\theta_{\varphi}\big(x_1,u\big)-D_u\theta_{\varphi}\big(x_2,u\big)\|\leqslant \widetilde{C}\|x_1-x_2\|.
	\end{equation*}
	By Hypothesis 5.1, item 2 holds. Furthermore, by Hypothesis $\textbf{H}_1$ we have that
	\begin{align*}
		\|\mu_X\big(x_1,u_1\big)-\mu_X\big(x_2,u_2\big)\|&\leqslant \|\mu_X\big(x_1,u_1\big)-\mu_X\big(x_2,u_1\big)\|+\|\mu_X\big(x_2,u_1\big)-\mu_X\big(x_2,u_2\big)\|\\
		&\leqslant C\|x_1-x_2\|+\|\mu_X\big(x_2,u_1\big)-\mu_X\big(x_2,u_2\big)\|
	\end{align*}
	where $C>0$ is a constant independent of $u\in \mathbb{U}$. Applying the mean value theorem with respect to the control coordinate to the second term of the above inequality, we have that 
	\begin{equation*}
		\|\mu_X\big(x_1,u_1\big)-\mu_X\big(x_2,u_2\big)\|\leqslant C\|x_1-x_2\|+\|D_u\mu_X\big(\widetilde{x},\widetilde{u}\big)\|\|u_1-u_2\|
	\end{equation*}
	where, by an abuse of notation, the pair $(\tilde{x},\tilde{u})$ also represents a point belonging to the line segment given by the points $(x_2,u_1)$ and $(x_2,u_2)$.  From Hypothesis 5.1 and the compactness of $\mathbb{X}\times \mathbb{U}$, there is a constant $\widetilde{C}>0$ such that 
	\begin{equation*}
		\|\mu_X\big(x_1,u_1\big)-\mu_X\big(x_2,u_2\big)\|\leqslant C\|x_1-x_2\|+\widetilde{C}\|u_1-u_2\|.
	\end{equation*}
	Therefore, when we add Hypothesis 5.1, the drift $\mu_X$ becomes Lipschitz continuous in the control coordinate.  The same reasoning applies to the dispersion $\sigma_X$ and the operation cost $L$. Thus, using Hypotheses 5.1 and 5.2 and Lemma \ref{lm:Lipschitz_controls} we obtain
	\begin{align*}
		\|\mu^{u^{*}}_X\big(x_1\big)-\mu^{u^{*}}_X\big(x_2\big)\|&=\|\mu_X\big(x_1,u^{*}\big(x_1\big)\big)-\mu_X\big(x_2,u^{*}\big(x_2\big)\big)\|\leqslant C\|x_1-x_2\|+\widetilde{C}\|u^{*}\big(x_1\big)+u^{*}\big(x_2\big)\|\\
		&\leqslant C\|x_1-x_2\|+\widetilde{C}\frac{C_0}{\varkappa}\|x_2-x_1\|=\Bigg(C+\widetilde{C}\frac{C_0}{\varkappa}\Bigg)\|x_1-x_2\|.
	\end{align*} 
	That is
	\begin{equation*}
		\|\mu^{u^{*}}_X\big(x_1\big)-\mu^{u^{*}}_X\big(x_2\big)\|\leqslant \Bigg(C+\widetilde{C}\frac{C_0}{\varkappa}\Bigg)\|x_1-x_2\|.
	\end{equation*}
	The same applies to the dispersion and the operation cost. 
	
	An important observation is that the examples provided in Section \S\ref{sec:examples} satisfy Hypotheses 5.1 and 5.2. We emphasize, however, that these conditions are merely sufficient to ensure the applicability of Lemma \ref{lm:Lipschitz_controls} to the function $\theta_{\varphi}$ defined in \eqref{eq:function_theta_varphi}. Therefore, in Theorem \ref{tm:existence_viscosity_solution_pde_hjb}, we will work under the following general hypothesis:
	
	\begin{itemize}
		\item[$\textbf{H}_4$] Structural assumptions on the drift, dispersion, and running cost functions that guarantee the validity of Lemma \ref{lm:Lipschitz_controls} for the function defined in \eqref{eq:function_theta_varphi}.
	\end{itemize}
\end{remark}
\begin{theorem}
	\label{tm:existence_viscosity_solution_pde_hjb}
	Assume that hypotheses $\emph{\textbf{H}}_1$ through $\emph{\textbf{H}}_4$ are satisfied. Then, the optimal value function defined in \eqref{eq:optimal_value_function} is a viscosity solution of the boundary value problem \eqref{eq:pde_hamilton_jacobi_bellman}–\eqref{eq:von_Neumann_boundary_condition}.. 
\end{theorem} 

The crucial step in the proof of Theorem \ref{tm:existence_viscosity_solution_pde_hjb} lies in verifying that the boundary inequalities \eqref{eq:boundary_viscosity_subsolution}–\eqref{eq:boundary_viscosity_supersolution} are satisfied. Indeed, within the interior of the domain $\mathbb{X}$, the viscosity sub- and supersolution properties follow from standard arguments, as detailed, for instance, in \cite[Chapter 4]{Pham}. Moreover, by Theorem \ref{tm:continuity_optimal_value_function}, the optimal value function defined in \eqref{eq:optimal_value_function} is continuous.
\begin{proof}[Proof: Viscosity Subsolution]
	The argument is by contradiction. Assume that there exists $\varphi\in \text{C}^{\infty}\big(\mathbb{X}\big)$ such that the application $x\mapsto \big(v^{\beta}-\varphi\big)\big(x\big)$ has a local maximum at the point $x_0\in\partial \mathbb{X}$ with $v^{\beta}\big(x_0\big)=\varphi\big(x_0\big)$ and
	\begin{equation*}
		F^{\beta}\Big(x_0,\varphi\big(x_0\big),D_x\varphi\big(x_0\big), D^2_x\varphi\big(x_0\big)\Big)\wedge\Gamma\Big(x_0,D_x\varphi\big(x_0\big)\Big)>0.
	\end{equation*} 
	By the continuity of $F^{\beta}$, $\Gamma$, $D^2_x\varphi$, $D_x\varphi$ and $\varphi$, there exist $\epsilon>0$ and $\delta>0$ such that 
	\begin{equation*}
		F^{\beta}\Big(x,\varphi\big(x\big),D_x\varphi\big(x\big), D^2_x\varphi\big(x\big)\Big)\geqslant\epsilon>0\quad\text{and}\quad\Gamma\Big(x, D_x\varphi\big(x\big)\Big)\geqslant\epsilon>0\quad \forall x\in \overline{B\big(x_0,\delta\big)}\cap \mathbb{X}.
	\end{equation*}
	By the definition of $F^{\beta}$, the first inequality implies that
	\begin{equation*}
		\beta\varphi\big(x\big)+\sup_{u\in \mathbb{U}}\bigg\{-\mathcal{L}^{u}_{X}\varphi\big(x\big)-L\big(x,u\big)\bigg\}\geqslant\epsilon>0\quad\forall x\in \overline{B\big(x_0,\delta\big)}\cap \mathbb{X}
	\end{equation*}
	where $\mathcal{L}^{u}_{X}$ is the differential operator \eqref{eq:auxiliary_equation_08}. Since the set $\mathbb{U}$ is compact and the fields $\mu_{X}$ and $\sigma_{X}$ are continuous, there exists $\widetilde{u}^{\epsilon}:\overline{B\big(x_0,\delta\big)}\cap \mathbb{X}\to \mathbb{U}$ such that 
	\begin{equation*}
		\beta\varphi\big(x\big)-\mathcal{L}^{\widetilde{u}^{\epsilon}}_{X}\varphi\big(x\big)-L\big(x,\widetilde{u}^{\epsilon}\big(x\big)\big)\geqslant\epsilon>0\quad\forall x\in \overline{B\big(x_0,\delta\big)}\cap \mathbb{X}.
	\end{equation*}
	From hypothesis $\textbf{H}_4$ it follows by Lemma \ref{lm:Lipschitz_controls} that $\widetilde{u}^{\epsilon}:\overline{B\big(x_0,\delta\big)}\cap \mathbb{X}\to \mathbb{U}$ is Lipschitz continuous. By \cite[Theorem 3.1]{EvansGariepy}, there exists a Lipschitz continuous function \( u^{\epsilon} : \mathbb{X} \to \mathbb{U} \) with the following property:
	\begin{equation*}
		u^{\epsilon}\big(x\big) = \widetilde{u}^{\epsilon}\big(x\big) \quad \forall x \in \overline{B\big(x_0, \delta\big)} \cap \mathbb{X}.
	\end{equation*}
	
	From hypotheses $\textbf{H}_1$ and $\textbf{H}_4$, it follows that there exists a single pair of processes $\big(X_{x_0}^{u^{\epsilon}}(t),l_{x_0}^{u^{\epsilon}}(t)\big)_{t\geqslant0}$ which is a strong solution of the SDE \eqref{eq:sde_reflection_control} when we use the control $u^{\epsilon}$ and the initial condition $x_0\in \partial \mathbb{X}$. From the equation \eqref{eq:dynamic_programming_principle_inf_inf} of the DPP (Theorem \ref{tm:dynamic_programming_principle}), it follows that 
	\begin{equation}
		\label{eq:auxiliary_inequality_13}
		v^{\beta}\big(x_0\big)\leqslant\inf_{\tau\in \mathcal{T}}\mathbb{E}\Bigg[\int_{0}^{\tau}e^{-\beta s}L\big(X^{u^{\epsilon}}_{x_0}(s), u^{\epsilon}(s)\big)\,ds +\int_{0}^{\tau}e^{-\beta s}h\big(X^{u^{\epsilon}}_{x_0}(s)\big)\,dl^{u^{\epsilon}}_{x_0}(s)+e^{-\beta\tau}v^{\beta}\big(X^{u^{\epsilon}}_{x_0}(\tau)\big)
		\Bigg].
	\end{equation}
	Now consider the stopping time $\tau^{\epsilon}:\Omega\to\mathbb{R}_{+}$ defined by 
	\begin{equation}
		\label{eq:stopping_time_tau_epsilon}
		\tau^{\epsilon}:=\inf\Bigg\{r: r\geqslant 0\hspace{0.10cm}\text{and}\hspace{0.10cm} X^{u^{\epsilon}}_{x_0}(r)\notin\overline{B\big(x_0,\delta\big)}\cap \mathbb{X}\Bigg\}.
	\end{equation}
	We claim that $\tau^{\epsilon}>0$ $\mathbb{P}$-a.s. Indeed, since $\big(X^{u^{\epsilon}}_{x_0}(t)\big)_{t\geqslant0}$ is a continuous process $\mathbb{P}$-a.s, therefore, if $X^{u^{\epsilon}}_{x_0}(r,\omega)\notin \overline{B\big(x_0,\delta\big)}\cap\mathbb{X}$ then there exists $\gamma^{\epsilon}(\omega)>0$ such that $r>\gamma^{\epsilon}(\omega)$. This implies
	\begin{equation*}
		\inf\Bigg\{r: r\geqslant 0\hspace{0.10cm}\text{and}\hspace{0.10cm}X^{u^{\epsilon}}_{x_0}(r,\omega)\notin\overline{B\big(x_0,\delta\big)}\cap\mathbb{X}\Bigg\}\geqslant\gamma^{\epsilon}(\omega)>0.
	\end{equation*}
	Therefore, by the definition of $\tau^{\epsilon}$, it follows that $\tau^{\epsilon}>0$ $\mathbb{P}$-a.s. On the other hand, by Itô's formula applied to the function $(r,x)\mapsto e^{-\beta r}\varphi\big(x\big)$, we obtain
	\begin{align*}
		e^{-\beta \tau^{\epsilon}}\varphi\big(X^{\widetilde{u}^{\epsilon}}_{x_0}(\tau^{\epsilon})\big)=\varphi\big(x_0\big)-&\int_{0}^{\tau^{\epsilon}}e^{-\beta s}\beta\varphi\big(X^{\widetilde{u}^{\epsilon}}_{x_0}(s)\big)\,ds\\
		+&\int_{0}^{\tau^{\epsilon}}e^{-\beta s}\bigg\langle D_{x}\varphi\big(X^{\widetilde{u}^{\epsilon}}_{x_0}(s)\big),\mu_X\big(X^{\widetilde{u}^{\epsilon}}_{x_0}(s),\widetilde{u}^{\epsilon}(s)\big)\bigg\rangle \,ds\\
		-&\int_{0}^{\tau^{\epsilon}}e^{-\beta s}\bigg\langle D_{x}\varphi\big(X^{\widetilde{u}^{\epsilon}}_{x_0}(s)\big), D_x\phi\big(X^{\widetilde{u}^{\epsilon}}_{x_0}(s)\big)\bigg\rangle\, dl^{\widetilde{u}^{\epsilon}}_{x_0}(s)\\
		+&\int_{0}^{\tau^{\epsilon}}e^{-\beta s}\frac{1}{2}\text{Tr}\Big(D^2_{x}\varphi\big(X^{\widetilde{u}^{\epsilon}}_{x_0}(s)\big)\big[\sigma_{X}\sigma^{\top}_{X}\big]\big(X^{\widetilde{u}^{\epsilon}}_{x_0}(s), \widetilde{u}^{\epsilon}(s)\big)\Big)\,ds\\
		+&\int_{0}^{\tau^{\epsilon}}\sum_{i=1}^{d}\sum_{l=1}^{d_W}e^{-\beta s}\partial_{x_i}\varphi\big(X^{\widetilde{u}^{\epsilon}}_{x_0}(s)\big)\sigma_{X, il}\big(X^{\widetilde{u}^{\epsilon}}_{x_0}(s),\widetilde{u}^{\epsilon}(s)\big)\,dW^{(l)}(s).
	\end{align*}
	Since the last term is a martingale, applying the expectation to both sides of the above equality follows that
	\begin{align*}
		\mathbb{E}\Big[e^{-\beta \tau^{\epsilon}}\varphi\big(X^{\widetilde{u}^{\epsilon}}_{x_0}(\tau^{\epsilon})\big)\Big]=\varphi\big(x_0\big)+\mathbb{E}\Bigg[- &\int_{0}^{\tau^{\epsilon}}e^{-\beta s}\beta\varphi\big(X^{\widetilde{u}^{\epsilon}}_{x_0}(s)\big)\,ds\\
		+&\int_{0}^{\tau^{\epsilon}}e^{-\beta s}\bigg\langle D_{x}\varphi\big(X^{\widetilde{u}^{\epsilon}}_{x_0}(s)\big),\mu_X\big(X^{\widetilde{u}^{\epsilon}}_{x_0}(s),\widetilde{u}^{\epsilon}(s)\big)\bigg\rangle \,ds\\
		-&\int_{0}^{\tau^{\epsilon}}e^{-\beta s}\bigg\langle D_{x}\varphi\big(X^{\widetilde{u}^{\epsilon}}_{x_0}(s)\big), D_x\phi\big(X^{\widetilde{u}^{\epsilon}}_{x_0}(s)\big)\bigg\rangle\, dl^{\widetilde{u}^{\epsilon}}_{x_0}(s)\\
		+&\int_{0}^{\tau^{\epsilon}}e^{-\beta s}\frac{1}{2}\text{Tr}\Big(D^2_{x}\varphi\big(X^{\widetilde{u}^{\epsilon}}_{x_0}(s)\big)\big[\sigma_{X}\sigma^{\top}_{X}\big]\big(X^{\widetilde{u}^{\epsilon}}_{x_0}(s), \widetilde{u}^{\epsilon}(s)\big)\Big)\,ds\Bigg].
	\end{align*}
	By definition $\mathcal{L}^{u}_{X}$, we can rewrite the above equality as
	\begin{align*}
		\varphi\big(x_0\big)=\mathbb{E}\Bigg[e^{-\beta \tau^{\epsilon}}\varphi\big(X^{\widetilde{u}^{\epsilon}}_{x_0}(\tau^{\epsilon})\big)-&\int_{0}^{\tau^{\epsilon}}e^{-\beta s}\big(-\beta\varphi+\mathcal{L}^{\widetilde{u}^{\epsilon}}_{X}\varphi\big)\big(X^{\widetilde{u}^{\epsilon}}_{x_0}(s),\widetilde{u}^{\epsilon}(s)\big)\,ds+\int_{0}^{\tau^{\epsilon}}e^{-\beta s}\bigg\langle D_{x}\varphi\big(X^{\widetilde{u}^{\epsilon}}_{x_0}(s)\big), D_x\phi\big(X^{\widetilde{u}^{\epsilon}}_{x_0}(s)\big)\bigg\rangle\, d l^{\widetilde{u}^{\epsilon}}_{x_0}(s)\Bigg].
	\end{align*}
	Denoting $\vartheta:=\mathbb{E}\big[1-e^{-\beta \tau^{\epsilon}}\big]>0$ and recalling that $v^{\beta}\big(x_0\big)=\varphi\big(x_0\big)$, it follows from the inequality \eqref{eq:auxiliary_inequality_13} that
	\begin{align*}
		0=\varphi\big(x_0\big)-v^{\beta}\big(x_0\big)\geqslant\mathbb{E}\Bigg[e^{-\beta\tau^{\epsilon}}\bigg[\varphi\big(X^{\widetilde{u}^{\epsilon}}_{x_0}(\tau^{\delta})\big)-v^{\beta}\big(X^{\widetilde{u}^{\epsilon}}_{x_0}(\tau^{\epsilon})\big)\bigg]&+\int_{0}^{\tau^{\epsilon}}e^{-\beta s}\big(\beta\varphi-\mathcal{L}^{\widetilde{u}^{\epsilon}}_{X}\varphi-L\big)\big(X^{\widetilde{u}^{\epsilon}}_{x_0}(s),\widetilde{u}^{\epsilon}(s)\big)\,ds\\
		&+\int_{0}^{\tau^{\epsilon}}e^{-\beta s}\big(\big\langle D_{x}\varphi,D_x\phi\big\rangle-h\big)\big(X^{\widetilde{u}^{\epsilon}}_{x_0}(s)\big)\,dl^{\widetilde{u}^{\epsilon}}_{x_0}(s)\Bigg].
	\end{align*}
	Now, given that 
	\begin{equation*}
		\big(\beta\varphi-\mathcal{L}^{\widetilde{u}^{\epsilon}}_{X}\varphi-L\big)\big(X^{\widetilde{u}^{\epsilon}}_{x_0}(s),\widetilde{u}^{\epsilon}(s)\big)\geqslant\epsilon\quad\text{and}\quad \big(\big\langle D_{x}\varphi,D_x\phi\big\rangle-h\big)\big(X^{\widetilde{u}^{\epsilon}}_{x_0}(s)\big)\geqslant \epsilon.
	\end{equation*}
	For all $s\in [0,\tau^{\epsilon}]$ it follows that
	\begin{align*}
		0=\varphi\big(x_0\big)-v^{\beta}\big(x_0\big)&\geqslant
		\mathbb{E}\Bigg[e^{-\beta\tau^{\epsilon}}\bigg[\varphi\big(X^{\widetilde{u}^{\epsilon}}_{x_0}(\tau^{\epsilon})\big)-v^{\beta}\big(X^{\widetilde{u}^{\epsilon}}_{x_0}(\tau^{\epsilon})\big)\bigg]+\epsilon\int_{0}^{\tau^{\epsilon}}e^{-\beta s}\,ds+\epsilon\int_{0}^{\tau^{\epsilon}}e^{-\beta s}\, dl^{\widetilde{u}^{\epsilon}}_{x_0}(s)\Bigg]\\
		&\geqslant\epsilon\mathbb{E}\Bigg[\int_{0}^{\tau^{\epsilon}}e^{-\beta s}ds+\int_{0}^{\tau^{\delta}}e^{-\beta s}\, d l^{\widetilde{u}^{\epsilon}}_{x_0}(s)\Bigg]\\
		&\geqslant\epsilon\mathbb{E}\Bigg[\int_{0}^{\tau^{\epsilon}}e^{-\beta s}\,ds\Bigg]=\epsilon\mathbb{E}\Bigg[\frac{1-e^{-\beta \tau^{\epsilon}}}{\beta}\Bigg]\\
		&=\frac{\epsilon}{\beta}\mathbb{E}\big[1-e^{-\beta \tau^{\epsilon}}\big]\geqslant\frac{\epsilon\vartheta}{\beta}>0.
	\end{align*}
	In other words,
	\begin{equation*}
		0=\varphi\big(x_0\big)-v^{\beta}\big(x_0\big)>0.
	\end{equation*}
	A contradiction, and therefore,
	\begin{equation*}
		\Gamma\Big(x_0,D_{x}\varphi\big(x_0\big)\Big)\wedge F^{\beta}\Big(x_0,\varphi\big(x_0\big),D_{x}\varphi\big(x_0\big),D^2_{x}\varphi\big(x_0\big)\Big) \leqslant0.
	\end{equation*}
	Therefore, the optimal value function \eqref{eq:optimal_value_function} is a viscosity subsolution of the PDE \eqref{eq:pde_hamilton_jacobi_bellman}-\eqref{eq:von_Neumann_boundary_condition}.
\end{proof}
\begin{proof}[Proof: Viscosity Supersolution]
	The argument is by contradiction. Assume that there exists $\varphi\in \text{C}^{\infty}\big(\mathbb{X}\big)$ such that the application $x\mapsto \big(v^{\beta}-\varphi\big)\big(x\big)$ has a local minimum at the point $x_0\in\partial \mathbb{X}$ with $v^{\beta}\big(x_0\big)=\varphi\big(x_0\big)$ and
	\begin{equation*}
		F^{\beta}\Big(x_0,\varphi\big(x_0\big),D_x\varphi\big(x_0\big), D^2_x\varphi\big(x_0\big)\Big)\wedge\Gamma\Big(x_0,D_x\varphi\big(x_0\big)\Big)<0.
	\end{equation*} 
	By the continuity of $F^{\beta}$, $\Gamma$, $D^2_x\varphi$, $D_x\varphi$ and $\varphi$, there exists $\epsilon>0$ and $\delta>0$ such that 
	\begin{equation*}
		F^{\beta}\Big(x,\varphi\big(x\big),D_{x}\varphi\big(x\big),D^2_{x}\varphi\big(x\big)\Big)\leqslant-\epsilon<0\quad\text{and}\quad\Gamma\Big(x,D_{x}\varphi\big(x\big)\Big)\leqslant-\epsilon<0\quad \forall x\in \overline{B\big(x_0,\delta\big)}\cap\mathbb{X}.
	\end{equation*}
	By the definition of $F^{\beta}$, the first inequality above means that
	\begin{equation*}
		\beta \varphi\big(x\big)-\inf_{u\in \mathbb{U}}\bigg\{\mathcal{L}^{u}_X\varphi\big(x\big)+L\big(x,u\big)\bigg\}\leqslant -\epsilon<0.
	\end{equation*}
	This implies, for every $u\in \mathbb{U}$, that
	\begin{equation*}
		\beta \varphi\big(x\big)-\mathcal{L}^{u}_X\varphi\big(x\big)-L\big(x,u\big)\leqslant -\epsilon<0\quad \forall x\in \overline{B\big(x_0,\delta\big)}\cap\mathbb{X}.
	\end{equation*}
	On the other hand, consider the equation \eqref{eq:dynamic_programming_principle_inf_sup} of the DPP (Theorem \ref{tm:dynamic_programming_principle}). Given $\rho>0$, there is a control $u^{\rho}\in\mathcal{U}$ such that 
	\begin{equation}
		\label{eq:auxiliary_inequality_dynamic_programming}
		v^{\beta}\big(x_0\big)+\frac{\epsilon\rho}{2}\geqslant\sup_{\tau\in \mathcal{T}}\mathbb{E}\Bigg[\int_{0}^{\tau}e^{-\beta s}L\big(X^{u^{\rho}}_{x_0}(s),u^{\rho}(s)\big)\,ds+\int_{0}^{\tau}e^{-\beta s}h\big(X^{u^{\rho}}_{x_0}(s)\big)\,dl^{u^{\rho}}_{x_0}(s)
		+e^{-\beta \tau}v^{\beta}\big(X^{u^{\rho}}_{x_0}(\tau)\big)\Bigg]
	\end{equation}
	with $\big(X^{u^{\rho}}_{x_0}(t),l^{u^{\rho}}_{x_0}(t)\big)_{t\geq 0}$ the unique strong solution of the SDE \eqref{eq:sde_reflection_control} associated with the control $u^{\rho}$ and the initial condition $x_0$.
	From the inequality \eqref{eq:auxiliary_inequality_dynamic_programming}, we have for all $\tau\in\mathcal{T}$, that
	\begin{equation}
		\label{eq:auxiliary_inequality_dynamic_programming_01}
		v^{\beta}\big(x_0\big)+\frac{\epsilon\rho}{2}\geqslant\mathbb{E}\Bigg[\int_{0}^{\tau}e^{-\beta s}L\big(X^{u^{\rho}}_{x_0}(s),u^{\rho}(s)\big)\,ds +\int_{0}^{\tau}e^{-\beta s}h\big(X^{u^{\rho}}_{x_0}(s)\big)\,dl^{u^{\rho}}_{x_0}(s)+e^{-\beta \tau}v^{\beta}\big(X^{u^{\rho}}_{x_0}(\tau)\big)\Bigg].
	\end{equation}
	On the other hand, consider the family of stopping times $\big(\theta^{\rho}\big)_{\rho>0}$ defined by 
	\begin{equation}
		\label{eq:supersolution_stop_time}
		\theta^{\rho}:=\rho\wedge\tau^{\rho}
	\end{equation}
	where the stopping time $\tau^{\rho}$ is defined in the analogous way to the stopping time  $\eqref{eq:stopping_time_tau_epsilon}$. We state that: 
	\begin{equation}
		\label{eq:auxiliary_equation_09}
		\lim_{\rho\to0}\frac{1}{\rho}\mathbb{E}\big[ \rho\wedge\tau^{\rho}\big]=1.
	\end{equation}
	In fact, according to Tchebyshev's inequality, we have that
	\begin{equation*}
		\mathbb{P}\big(\tau^{\rho}\leqslant \rho\big)\leqslant \mathbb{P}\Big(\sup_{r\in[0,\rho]}\|X^{u^{\rho}}_{x_0}(r)-x_0\|> \delta\Big)\leqslant\frac{1}{\delta^2}\mathbb{E}\Bigg[\sup_{r\in[0,\rho]}\|X^{u^{\rho}}_{x_0}(r)-x_0\|^2\Bigg].
	\end{equation*}
	By inequality 1.18 in \cite[Chapter 1]{Pham}
	\begin{equation*}
		\mathbb{E}\Bigg[\sup_{r\in[0,\rho]}\|X^{u^{\rho}}_{x_0}(r)-x_0\|^2\Bigg]\leqslant C\big(1+\|x_0\|^2\big)\rho e^{C\rho}
	\end{equation*}
	for strong solutions of SDEs, where $C>0$ is a constant (depending on the Lipschitz constant) and combining this inequality with the previous one, we get
	\begin{equation*}
		\mathbb{P}\big(\tau^{\rho}\leqslant \rho\big)\leqslant \frac{C}{\delta^2}\big(1+\|x_0\|^2\big)\rho e^{C\rho}.
	\end{equation*}
	Making $\rho\to 0$, it follows that 
	\begin{equation*}
		\lim_{\rho\to0}\mathbb{P}\big(\tau^{\rho}\leqslant \rho\big)=0.
	\end{equation*}
	On the other hand, by the Markov inequality, we get
	\begin{equation*}
		\mathbb{P}\big(\tau^{\rho}>\rho\big)=\mathbb{P}\big(\rho\wedge\tau^{\rho}>\rho\big)\leqslant \frac{1}{\rho}\mathbb{E}\big[ \rho\wedge\tau^{\rho}\big]\leqslant\frac{\rho}{\rho}=1.
	\end{equation*}
	Making $\rho\to 0$, we get \eqref{eq:auxiliary_equation_09}.
	
	Returning to the main result, we combine the stopping time \eqref{eq:supersolution_stop_time} with the inequality \eqref{eq:auxiliary_inequality_dynamic_programming_01}, to obtain
	\begin{equation*}
		v^{\beta}\big(x_0\big)+\frac{\epsilon\rho}{2}\geqslant\mathbb{E}\Bigg[\int_{0}^{\theta^{\rho}}e^{-\beta s}L\big(X^{u^{\rho}}_{x_0}(s),u^{\rho}(s)\big)\,ds+\int_{0}^{\theta^{\rho}}e^{-\beta s}h\big(X^{u^{\rho}}_{x_0}(s)\big)\,dl^{u^{\rho}}_{x_0}(s) +e^{-\beta \theta^{\rho} }v^{\beta}\big(X^{u^{\rho}}_{x_0}(\theta^{\rho})\big)\Bigg].
	\end{equation*}
	From Itô's formula applied to the function $(r,x)\mapsto e^{-\beta r}\varphi\big(x\big)$, it follows by arguments analogous to those used in the proof of viscosity subsolution, that 
	\begin{equation*}
		\varphi\big(x_0\big)=\mathbb{E}\Bigg[e^{-\beta \theta^{\rho}}\varphi\big(X^{u^{\rho}}_{x_0}(\theta^{\rho})\big)-\int_{0}^{\theta^{\rho}}e^{-\beta s}\big(-\beta\varphi+\mathcal{L}^{u^{\rho}}_{X}\varphi\big)\big(X^{u^{\rho}}_{x_0}(s)\big)\,ds+\int_{0}^{\theta^{\rho}}e^{-\beta s}\bigg\langle D_{x}\varphi\big(X^{u^{\rho}}_{x_0}(s)\big),D_x\phi\big(X^{u^{\rho}}_{x_0}(s)\big)\bigg\rangle \,dl^{u^{\rho}}_{x_0}(s)\Bigg].
	\end{equation*}
	Recalling that $v^{\beta}(x_0)=\varphi(x_0)$, we get
	\begin{align*}
		0=\varphi\big(x_0\big)-v^{\beta}\big(x_0\big)\leqslant\mathbb{E}\Bigg[e^{-\beta \theta^{\rho}}\bigg[\varphi\big(X^{u^{\rho}}_{x_0}(\theta^{\rho})\big)-v^{\beta}\big(X^{u^{\rho}}_{x_0}(\theta^{\rho})\big)\bigg]+&\int_{0}^{\theta^{\rho}}e^{-\beta s}\big(\beta\varphi-\mathcal{L}^{u^{\rho}}_{X}\varphi-L\big)\big(X^{u^{\rho}}_{x_0}(s),u^{\rho}(s)\big)\,ds\\
		+&\int_{0}^{\theta^{\rho}}e^{-\beta s}\big(\big\langle D_{x}\varphi,D_x\phi\big\rangle-h\big)\big(X^{u^{\rho}}_{x_0}(s)\big)\, dl^{u^{\rho}}_{x_0}(s)\Bigg]+\frac{\epsilon\rho}{2}.
	\end{align*}
	Now, given that
	\begin{equation*}
		\big(\beta\varphi-\mathcal{L}^{u^{\rho}}_{X}\varphi-L\big)\big(X^{u^{\rho}}_{x_0}(s),u^{\rho}(s)\big)\leqslant-\epsilon\quad\text{and}\quad \big(\big\langle D_{x}\varphi,D_x\phi \big\rangle-h\big)\big(X^{u^{\rho}}_{x_0}(s)\big)\leqslant -\epsilon
	\end{equation*}
	for all $s\in [0,\theta^{\rho}]$, we have that
	\begin{align*}
		0=\varphi\big(x_0\big)-v^{\beta}\big(x_0\big)\leqslant&
		\mathbb{E}\Bigg[e^{-\beta \theta^{\rho}}\bigg[\varphi\big(X^{u^{\rho}}_{x_0}(\theta^{\rho})\big)-v^{\beta}\big(X^{u^{\rho}}_{x_0}(\theta^{\rho})\big)\bigg]-\epsilon\int_{0}^{\theta^{\rho}}e^{-\beta s}\,ds-\epsilon\int_{0}^{\theta^{\rho}}e^{-\beta s}\, dl^{u^{\rho}}_{x_0}(s)\Bigg]+\frac{\epsilon\rho}{2}\\
		\leqslant&-\epsilon\mathbb{E}\Bigg[\int_{0}^{\theta^{\rho}}e^{-\beta s}\,ds+\int_{0}^{\theta^{\rho}}e^{-\beta s}\, dl^{u^{\rho}}_{x_0}(s)\Bigg]+\frac{\epsilon\rho}{2}\leqslant-\epsilon\mathbb{E}\Bigg[\int_{0}^{\theta^{\rho}}e^{-\beta s}\,ds\Bigg]+\frac{\epsilon\rho}{2}\\
		\leqslant&-\epsilon e^{-\beta \rho}\mathbb{E}\big[\theta^{\rho}\big]+\frac{\epsilon\rho}{2}=-\epsilon e^{-\beta \rho}\mathbb{E}\big[\rho\wedge\tau^{\rho}\big]+\frac{\epsilon\rho}{2}.
	\end{align*}
	In other words,
	\begin{equation*}
		0\leqslant -\epsilon e^{-\beta \rho}\mathbb{E}\big[\rho\wedge\tau^{\rho}\big]+\frac{\epsilon\rho}{2}.
	\end{equation*}
	Dividing the above inequality by $\rho>0$, we get
	\begin{equation*}
		0\leqslant -\epsilon \frac{e^{-\beta \rho}}{\rho}\mathbb{E}\big[\rho\wedge\tau^{\rho}\big]+\frac{\epsilon}{2}.
	\end{equation*}
	Making $\rho\to 0$, it follows that 
	\begin{equation*}
		0\leqslant -\epsilon+\frac{\epsilon}{2}=-\frac{\epsilon}{2}<0.
	\end{equation*}
	A contradiction and so we get
	\begin{equation*}
		\Gamma\Big(x_0,D_{x}\varphi\big(x_0\big)\Big)\lor F^{\beta}\Big(x_0,\varphi\big(x_0\big),D_{x}\varphi\big(x_0\big),D^2_{x}\varphi\big(x_0\big)\Big) \geqslant0.
	\end{equation*}
	Therefore, the optimal value function \eqref{eq:optimal_value_function} is a viscosity supersolution of the PDE \eqref{eq:pde_hamilton_jacobi_bellman}.
\end{proof}

 The result stated below was originally established in \cite{Crandall}, using a definition of viscosity solution based on the notions of superjets and subjets, which differs from the definition adopted here (see Definition \ref{df:viscosity_solutions}). Nevertheless, it is known that these formulations are equivalent; see, for instance, \cite[Chapter 5]{Soner} or \cite[Appendix D]{Pardoux} for a detailed discussion.

\begin{theorem}[Comparison theorem for problems with Von
	Neumann boundary condition \protect{\cite[Section 7]{Crandall}}]
	\label{tm:unicity_PVC_Von_Neumann}
	Assume that the structural conditions described in \cite[Section 7]{Crandall} are satisfied, and that the problem \eqref{subeq:pde_hjb_and_von_Neumann_boundary_condition} admits a viscosity subsolution $w$ and a viscosity supersolution $v$. Then, $w\big(x\big)\leqslant v\big(x\big)$ for all $x\in\mathbb{X}$.
\end{theorem} 

Since Theorem \ref{tm:existence_viscosity_solution_pde_hjb} establishes that the optimal value function \eqref{eq:optimal_value_function} is a viscosity solution of the boundary value problem \eqref{subeq:pde_hjb_and_von_Neumann_boundary_condition}, Theorem \ref{tm:unicity_PVC_Von_Neumann} ensures that this solution is unique.
	
\subsection{Examples}
\label{sec:examples}
In this section we present two examples that implement what we have developed in the previous sections. In the first, we have a control problem for which the HBJ equation is semilinear, so we can explicitly calculate the optimal control as a function of the gradient of the optimal value function. In the second, we have a control problem for which the HJB equation is totally nonlinear.

In both cases, we will consider the following space for the slow variable
\begin{equation}
	\label{eq:domain_D_slow_variable}
	\mathbb{X}:=\big[-\alpha,\alpha\big].
\end{equation}
for $\alpha>0$. Consider the following function $\phi:\mathbb{R}\to\mathbb{R}$ defined by
\begin{equation}
	\label{eq:example_phi_function}
	\phi\big(x\big):=\frac{1}{2\alpha}e^{\alpha^2-x^2}\big(x^2-\alpha^2\big).
\end{equation}
Let's show that the set $\mathbb{X}$ has the following representation 
\begin{equation}
	\label{eq:domain_D_slow_variable_phi}
	\mathbb{X}=\Bigg\{x\in\mathbb{R}:\phi\big(x\big)\leqslant 0\Bigg\}
\end{equation}
in terms of the function \eqref{eq:example_phi_function}. To do this, let's look at some properties of the function \eqref{eq:example_phi_function}. First of all, note that
\begin{equation*}
	\lim_{x\to+\infty}\phi\big(x\big)=0\quad\text{and}\quad \lim_{x\to-\infty}\phi\big(x\big)=0.
\end{equation*}
Note that \eqref{eq:example_phi_function} is clearly $\text{C}^2\big(\mathbb{R}\big)$ with derivative given by
\begin{equation}
	\label{eq:example_phi_function_derivative}
	\frac{d\phi}{dx}\big(x\big)=\frac{x}{\alpha}e^{\alpha^2-x^2}\big(1+\alpha^2-x^2\big).
\end{equation}
At points $-\alpha$ and $\alpha$ we have that
\begin{equation*}
	\frac{d\phi}{dx}\big(\alpha\big)=1\quad \text{and}\quad \frac{d\phi}{dx}\big(-\alpha\big)=-1.
\end{equation*}
The critical points of $\phi$ are
\begin{equation*}
	\frac{d\phi}{dx}\big(0\big)=0,\quad \frac{d\phi}{dx}\big(\sqrt{1+\alpha^2}\big)=0\quad\text{and}\quad \frac{d\phi}{dx}\big(-\sqrt{1+\alpha^2}\big)=0.
\end{equation*}
where $0$ is the minimum point and $\sqrt{1+\alpha^2}$ and $-\sqrt{1+\alpha^2}$ are the maximum points.  Then we have that \eqref{eq:example_phi_function} is a bounded function. Furthermore, we have
\begin{equation*}
	\phi\big(x\big)<0\quad \forall x\in\big(-\alpha,\alpha\big),\quad \phi\big(-\alpha\big)=0\quad\text{and} \quad \phi\big(\alpha\big)=0.
\end{equation*}
We conclude that the set \eqref{eq:domain_D_slow_variable}  admits the representation \eqref{eq:domain_D_slow_variable_phi}.

The above development is generalized to the $d_X$-dimensional case when we consider $\mathbb{X}:=\text{B}\big(0,\alpha\big)$ with $\alpha>0$ and $\phi:\mathbb{R}^{d_X}\to\mathbb{R}$ given by 
\begin{equation*}
	\phi\big(x\big):=\frac{1}{2\alpha}e^{\alpha^2-\|x\|^2}\big(\|x\|^2-\alpha^2\big).
\end{equation*}
\subsubsection{Control Problem with Semilinear HJB Equation}
\label{subsec:control_problem_semilinear_hjb_equation}
 The stochastic dynamical system we will be working with is defined by:
	\begin{equation}
		\label{eq:dynamic_system_exemplo_01}
		dX_x(t)=\bigg[\theta_a X_x(t)-\theta_b u(t)\bigg]\,dt+\sigma_{X}X_x(t)\,dW(t)-D_x\phi\big(X_x(t)\big)\,dl_{x}(t),
	\end{equation}
where $\theta_a,\theta_b\in\mathbb{R}$ are parameters, as well as $\sigma_{X}>0$. The control $\big(u^{\epsilon}(t)\big)_{t\geqslant0}$ has as its state space the interval $\big[u_a,u_b\big]$ with $u_a,u_b\in\mathbb{R}$ and $u_a<u_b$. Furthermore, we consider $\mathcal{U}$ to be the set of progressively measurable processes with state space given by the interval $\big[u_a,u_b\big]$. Finally, the process $\big(l_{x}(t)\big)_{t\geqslant 0}$ is continuous, non-decreasing, with $l_{x}\big(0\big)=0$ and satisfies the condition
\begin{equation*}
	l_{x}(t)=\int_{0}^{t}\mathds{1}_{\partial \mathbb{X}}\big(X^{\epsilon}_x(t)\big)\,dl_{x}(s)\quad \mathbb{P}\text{-a.s}.
\end{equation*}
The drift and dispersion associated with the system \eqref{eq:dynamic_system_exemplo_01} are given by:
\begin{subequations}
	\begin{align}
		\label{eq:drift_example_01}
		\mu_{X}\big(x,u\big)&:=\theta_a x-\theta_b u,\\
		\label{eq:dispersion_example_01}
		\sigma_{X}\big(x,u\big)&:=\sigma_{X}x
	\end{align}
\end{subequations} 
Note that the fields \eqref{eq:drift_example_01} and \eqref{eq:dispersion_example_01} are bounded and Lipschitz continuous in the state variables $x$ in $\mathbb{X}\times\big[u_a,u_b\big]$ uniformly with respect to the control variable $u$. 

We now move on to the definitions of the operational and preventive costs. We consider the following function to be the operational cost
\begin{equation}
	\label{eq:operational_cost_example_01}
	L\big(x,u\big):=\big(\theta_d-u\big)^2
\end{equation}
with $\theta_d>0$ a parameter. This function is bounded and Lipschitz continuous in $\mathbb{X}\times\big[u_a,u_b\big]$ uniformly with respect to the control variable $u$. For the preventive cost we simply take
\begin{equation*}
	h\big(-\alpha\big):=\theta_e=:h\big(\alpha\big)
\end{equation*}
where $\theta_e\in\mathbb{R}$ is a parameter. 

Fixed $\beta>0$, we define the cost functional by
\begin{equation}
	\label{eq:operational_functional_example_01}
	J_{x}^{\beta}\big(u\big):=\mathbb{E}\Bigg[\int_{0}^{+\infty}e^{-\beta s}\big(\theta_d-u(s)\big)^2\,ds+\int_{0}^{+\infty}e^{-\beta s}\theta_e\,dl_{x}(s)\Bigg].
\end{equation}
Thus, the optimal value function is given by
\begin{equation}
	\label{eq:optimal_value_function_example_01}
	v^{\beta}\big(x\big):=\inf_{u\in \mathcal{U}}J_{x}^{\beta}\big(u\big)
\end{equation}
and by Theorem \ref{tm:continuity_optimal_value_function} in Subsection \S\ref{subsec:continuity_optimal_value_function}, the function \eqref{eq:optimal_value_function_example_01} is continuous. Using the principle of dynamic programming, proven in Subsection \S\ref{subsec:dynamic_programming_principle}, we obtain the following HJB equation 
\begin{subequations}
	\label{subeq:HJB_equation_boundary_condition_example_01}
	\begin{align}
		\label{eq:HJB_equation_exemplo_01}
		\beta v^{\beta}\big(x\big)-\mathcal{H}\Big(x,\partial_xv^{\beta}\big(x\big),\partial^2_{x^2}v^{\beta}\big(x\big)\Big)&=0\quad\forall x\in \mathbb{X},\\
		\label{eq:boundary_condition_example_01}
		\partial_{x}v^{\beta}\big(-\alpha\big)=-\theta_e\hspace{0.15cm}\text{and}\hspace{0.15cm} \partial_{x}v^{\beta}\big(\alpha\big)&=\theta_e.
	\end{align}
\end{subequations}
where the Hamiltonian is defined by
\begin{align*}
	\mathcal{H}\big(x,g_x,H_x\big):=&\frac{\sigma^2_{X}x^2}{2}H_x+\theta_axg_x+\min_{u\in [u_a,u_b]}\bigg\{u^2-\big(2\theta_d+\theta_bg_x\big)u\bigg\}+\theta_d^2.
\end{align*}
To obtain an explicit formula for the Hamiltonian \eqref{eq:HJB_equation_exemplo_01}, we must calculate the minimization 
\begin{equation*}
	u_{\eta}^{*}\big(g_x\big):=\arg\min_{u\in [u_a,u_b]}\bigg\{u^2-f_{\eta}\big(g_x\big)u\bigg\}
\end{equation*}
where $\eta:=\big(\theta_b,\theta_d,u_a,u_b\big)$ and \(f_{\eta}\big(g_x\big):=2\theta_d+\theta_bg_x\). Assuming $\theta_b\ne 0$, we obtain that
\begin{equation}
	\label{eq:control_function_example_01}
	u_{\eta}^{*}\big(g_x\big):=u_a \mathds{1}_{\mathbb{I}_1}\big(g_x\big)+\frac{1}{2}f_{\eta}\big(g_x\big)\mathds{1}_{\mathbb{I}_2}\big(g_x\big)+u_b\mathds{1}_{\mathbb{I}_3}\big(g_x\big)
\end{equation}
where
\begin{equation*}
	\mathbb{I}_1:=\Bigg(-\infty,\frac{2\big(u_a-\theta_d\big)}{\theta_b}\Bigg],\hspace{0.15cm} \mathbb{I}_2:=\Bigg[\frac{2\big(u_a-\theta_d\big)}{\theta_b},\frac{2\big(u_b-\theta_d\big)}{\theta_b}\Bigg]\hspace{0.15cm}\text{and}\hspace{0.15cm}\mathbb{I}_3:=\Bigg[\frac{2\big(u_b-\theta_d\big)}{\theta_b},+\infty\Bigg).
\end{equation*}
Note that $u_{\eta}^{*}$ is an poligonal function and therefore Lipschitz continuous. Furthermore, observe that if \(v^{\beta}:\mathbb{X}\mapsto \mathbb{R}\) is at least Lipschitz continuous, then by Rademacher’s Theorem \cite[Theorem 3.2]{EvansGariepy}, the derivative of \(v^{\beta}\) exists almost everywhere with respect to the Lebesgue measure and is measurable at every point where it is differentiable. Thus, by means of equation \eqref{eq:control_function_example_01}, we obtain the following Markovian control
\begin{equation}
	\label{eq:markovian_control_example_01}
	\widetilde{u}\big(x\big):=u_{\eta}^{*}\bigg(\partial_xv^{\beta}\big(x\big)\bigg).
\end{equation}

Observe that, due to the low regularity of the optimal value function \eqref{eq:optimal_value_function_example_01} --- even in the case where it is Lipschitz continuous --- it is not possible to apply classical control tools such as, for instance, the Verification Theorem to ensure the optimality of \eqref{eq:markovian_control_example_01}.

Finally, by Theorem \ref{tm:existence_viscosity_solution_pde_hjb}, the optimal value function \eqref{eq:optimal_value_function_example_01} is a viscosity solution of the PDE \eqref{eq:HJB_equation_exemplo_01} and, by Theorem \ref{tm:unicity_PVC_Von_Neumann}, it is the unique viscosity solution. 

\subsubsection{Control Problem with Fully Nonlinear HJB Equation}
\label{subsec:control_problem_fully_nonlinear_HJB_equation}
The stochastic dynamical system we will be working with is defined by:
	\begin{align}
		\label{eq:dynamic_system_exemplo_02}
		dX_x(t)=&\bigg[\theta_a X_x(t)-\theta_b u(t)\bigg]\,dt+\sigma_{X} \sqrt{u(t)}X_x(t)\,dW(t)-D_x\phi\big(X_x(t)\big)\,dl_{x}(t)
	\end{align}
The difference between the dynamical system \eqref{eq:dynamic_system_exemplo_01} shown previously and \eqref{eq:dynamic_system_exemplo_02} is that we include the control process in the diffusion term. 

Using the same operating cost as in example \ref{subsec:control_problem_semilinear_hjb_equation} (equation \eqref{eq:operational_cost_example_01}) as well as the same preventive boundary cost, we obtain the following HJB equation
\begin{subequations}
	\label{subeq:HJB_equation_boundary_condition_example_02}
	\begin{align}
		\label{eq:HJB_equation_exemplo_02}
		\beta v^{\beta}\big(x\big)-\mathcal{H}\Big(x,\partial_xv^{\beta}\big(x\big),\partial^2_{x^2}v^{\beta}\big(x\big)\Big)&=0\quad\forall x\in \mathbb{X},\\
		\label{eq:boundary_condition_example_02}
		\partial_{x}v^{\beta}\big(-\alpha\big)=-\theta_e\hspace{0.15cm}\text{and}\hspace{0.15cm} \partial_{x}v^{\beta}\big(\alpha\big)&=\theta_e.
	\end{align}
\end{subequations}
where the Hamiltonian is defined by
\begin{equation*}
	\mathcal{H}\Big(x,g_x,H_x\Big):=\min_{u\in [u_a,u_b]}\Bigg\{u^2-\bigg(2\theta_d+\theta_bg_x-\frac{1}{2}\sigma^2_{X}x^2H_x\bigg)u\Bigg\}+\theta_d^2.
\end{equation*}
To obtain an explicit formula for the Hamiltonian \eqref{eq:HJB_equation_exemplo_02}, we must calculate the minimization 
\begin{equation*}
	\widetilde{u}_{\eta}^{*}\big(x,g_x,H_x\big):=\arg\min_{u\in [u_a,u_b]}\Bigg\{u^2- \widetilde{f}_{\eta}\big(x,g_x,H_x\big)u\Bigg\}
\end{equation*}
where $\eta:=\big(\theta_b,\theta_d,u_a,u_b\big)$ and
\[
   \widetilde{f}_{\eta}\big(x,g_x,H_x\big):=f_{\eta}\big(g_x\big)-\frac{1}{2}\sigma^2_{X}x^2H_x.
\]
Due to the explicit dependence on both the Hessian matrix and the state variable \(x\), the analysis carried out in equation~\eqref{eq:control_function_example_01} becomes more intricate. Nevertheless, it is still possible to derive insights by treating \(x \in \mathbb{X}\) and \(H_x \in \mathbb{R}\) as free parameters, while considering \(g_x\) as a dependent variable. Under this framework, we define the control law as follows:
\begin{equation}
	\label{eq:control_function_example_02}
	\widetilde{u}_{\eta}^{*}\big(x,g_x,H_x\big):=u_a \mathds{1}_{\mathbb{I}_1\big\{x, H_x\big\}}\big(g_x\big)+\frac{1}{2}\widetilde{f}_{\eta}\big(x,g_x,H_x\big)\mathds{1}_{\mathbb{I}_2\big\{x, H_x\big\}}\big(g_x\big)+u_b\mathds{1}_{\mathbb{I}_3\big\{x, H_x\big\}}\big(g_x\big)
\end{equation}
where the indicator sets are given by:
\begin{align*}
	\mathbb{I}_1\big\{x, H_x\big\} := & \left(-\infty,\frac{2(u_a-\theta_d)}{\theta_b}+\frac{\sigma^2_Xx^2H_x}{2\theta_b}\right]; \\
	\mathbb{I}_2\big\{x, H_x\big\} := & \left[\frac{2(u_a-\theta_d)}{\theta_b}+\frac{\sigma^2_Xx^2H_x}{2\theta_b},\frac{2(u_b-\theta_d)}{\theta_b}+\frac{\sigma^2_Xx^2H_x}{2\theta_b}\right]; \\
	\mathbb{I}_3\big\{x, H_x\big\} := & \left[\frac{2(u_b-\theta_d)}{\theta_b}+\frac{\sigma^2_Xx^2H_x}{2\theta_b},+\infty\right).
\end{align*}
 Furthermore, when \(x = 0\), the expression in~\eqref{eq:control_function_example_02} reduces to that of equation~\eqref{eq:control_function_example_01}.

By Theorem \ref{tm:existence_viscosity_solution_pde_hjb}, the optimal value function \eqref{eq:optimal_value_function_example_01} is a viscosity solution of the PDE \eqref{eq:HJB_equation_exemplo_02} and, by Theorem \ref{tm:unicity_PVC_Von_Neumann}, it is the unique viscosity solution.
\bibliographystyle{siam}
\bibliography{references}
\end{document}